\documentclass{amsart}
\usepackage{amsmath,amssymb,enumerate,amscd,tikz,tikz-cd,url, color, array, multirow}
 
\usepackage[skip=2pt]{caption} 

\theoremstyle{plain}
\newtheorem{theo}[equation]{Theorem}
\newtheorem{lem}[equation]{Lemma}
\newtheorem{cor}[equation]{Corollary}
\newtheorem{prop}[equation]{Proposition}
\theoremstyle{definition}
\newtheorem{Ques}{Question}

\newtheorem{Def}[equation]{Definition}

\newtheorem{Rem}[equation]{Remark} 
\newtheorem{Not}[equation]{Notation}

\newcommand{\F}{{\mathbb F}}
    
\newcommand{\Q}{{\mathbb Q}}

\newcommand{\Z}{{\mathbb Z}}

\newcommand{\bfD}{{\bf D}}

\newcommand{\bfK}{{\bf K}}

\newcommand{\bfOd}{{\bf Od}}
\newcommand{\bfQE}{{\bf QE}}

\newcommand{\bfT}{{\bf T}}

\newcommand{\Mu}{\boldsymbol \mu}

\newcommand{\gc}{{\mathfrak c}}

\newcommand{\gf}{{\mathfrak f}}
\newcommand{\gF}{{\mathfrak F}}

\newcommand{\gp}{{\mathfrak p}}
\newcommand{\gq}{{\mathfrak q}}
\newcommand{\gP}{{\mathfrak P}}
\newcommand{\gQ}{{\mathfrak Q}}

\newcommand{\cD}{{\mathcal D}}
\newcommand{\cE}{{\mathcal E}}
\newcommand{\cF}{{\mathcal F}}
\newcommand{\cG}{{\mathcal G}}
\newcommand{\cH}{{\mathcal H}}
\newcommand{\cI}{{\mathcal I}}

\newcommand{\cO}{{\mathcal O}}

\newcommand{\cS}{{\mathcal S}}

\newcommand{\un}[1]{\underline{#1}}

\newcommand{\Ext}{\operatorname{Ext}}
\newcommand{\grad}{\operatorname{\mathbf{gr}}}

\newcommand{\rk}{\operatorname{rank}}
\newcommand{\Gal}{\operatorname{Gal}}
\newcommand{\Hom}{\operatorname{Hom}}

\newcommand{\ov}[1]{\overline{#1}}
\newcommand{\fdeg}[2]{[{#1}\!:\!{#2}]}
\newcommand{\vv}[1]{\vert {#1} \vert}
\newcommand{\lr}[1]{\langle{#1}\rangle}

\makeatletter
\renewcommand*\l@subsection{\@tocline{2}{0pt}{30pt}{0pt}{}}
\makeatother

\begin{document}
\title{Fields with no everywhere good abelian varieties}
\subjclass[2020]{Primary 11G10; \, Secondary 11S15, 14L15}
\keywords{Abelian varieties over number fields, everywhere good reduction, ramification bounds, group schemes, class field theory}
\date{26 October 2026}

\author[Armand Brumer]{Armand Brumer}
\address{Department of Mathematics \\
Fordham University \\
Bronx, NY 10458, USA}
\email{brumer@fordham.edu}

\author[Kenneth Kramer]{Kenneth Kramer}
\address{Department of Mathematics \\
Queens College and the Graduate Center, CUNY \\
Queens, NY 11367, USA}
\email{kkramer@qc.cuny.edu}

\begin{abstract}
We extend methods of Fontaine, Abrashkin and Schoof to obtain criteria determining number fields K over which no non-zero abelian variety with everywhere good reduction exists.  As an application, under the GRH,  we find 26847 such fields of various degrees up to 16.
\end{abstract}

\maketitle 

\tableofcontents

\numberwithin{equation}{section}

\section{Introduction}  \label{Intro}

Given a number field $K$, a finite set of places $S$ and a positive integer $g$, there are finitely many abelian varieties of dimension $g$ with good reduction outside $S$ by \cite{Fa,Zar}.  Fontaine \cite{Fo} and Abrashkin  \cite{Ab} independently proved that there is no non-zero abelian variety with everywhere good reduction over $\Q$. 

\begin{Def} We say that $K$ is a {\em Fontaine field} if there is no non-zero abelian variety over $K$ with everywhere good reduction. 
\end{Def}

Schoof \cite{Sch3} showed that for $f$ in $\{3,4,5,7,8,9,12,11,15\},$  the cyclotomic fields $\Q(\zeta_f)$ are Fontaine fields.  In \cite{Sch7a}, the same is true of the quadratic fields $\Q(\sqrt{d})$ with $d$ in $\{5,8,12,13,17,21\}$.  Those 16 number fields seem to be  the only known Fontaine fields, even under the assumption of GRH.  This paper extends the methods used by those authors. Assuming GRH, we verify that  many fields $K$ of degree $n\le16$ are Fontaine, as shown in Table \ref{punchline}. 

\begin{Ques} 
How does the number of Fontaine fields of degree $n$  vary with $n$? 
\end{Ques} 

All semistable abelian varieties over $\Q$ good outside a set of small primes were determined in \cite{BK1,BK2,Ca,Sch4,Sch5,Sch6}.  Schoof and Dembele  \cite{Sch1,Sch7,De,Sch2} show that for some real quadratic fields $K$ of discriminant at most 100 and for $\Q(\zeta_{20})$, there is only one isogeny class of simple abelian varieties good everywhere.

\begin{Def}\label{FoSc_def} 
Let $K$ be a number field and let $S$ be a finite set of primes of $K$. The pair $(K,S)$ is a {\em Schoof pair} if there are only finitely many absolutely simple  semistable abelian varieties over $K$ good outside $S$.  
\end{Def}

It is of interest to find Schoof pairs. It is more challenging to consider the following questions.  The second one occurred as Question 7.4 in \cite{Kha}.

\begin{Ques}  
Are there only finitely many Schoof pairs?  Could it be that all pairs $(K,S)$ are Schoof pairs? 
\end{Ques}

Recall that an abelian variety $A$ is good everywhere over $K$ if and only if its N\'eron model is a group scheme over the maximal order $\cO_K$ of $K$. Then $A[m]$ is a finite flat group scheme over $\cO_K$ for every integer $m \ne 0$. 

Let $\Delta_K$ denote the discriminant of $K$ and let $\delta_K=\Delta_K^{1/[K:\Q]}$ be its root discriminant. The two main tools available to decide whether $K$ is a Fontaine field are found in \cite{Fo,Odl}.  Let $E$ be a finite flat group scheme over $\cO_K$ of exponent 2.  The Fontaine bound on higher ramification in \cite{Fo} implies that the root discriminant of $L =K(E)$ is bounded by $\delta_L< 4 \delta_K$, where the inequality should be interpreted as strict divisibility. Odlyzko  then bounds the degree $\fdeg{L}{\Q}$ in terms of $\delta_K.$  Currently, to know that there are only finitely many simple finite flat group schemes of exponent 2 over $\cO_K$ requires the  Odlyzko bounds.  However, one might consider the following easier question, with no restriction on the degree of $K$.  See for instance \cite{BK3, BK4}.

\begin{Ques}  \label{QC}
Let $K$ be a  number field, $S$ a finite set of primes of $K$ prime to $\ell$ and $\cO_S$  the ring of $S$-integers  of $K$.  Given a collection $\mathcal{E}$ of simple finite flat group  schemes of exponent $\ell$ over $\cO_S$, what are the semistable abelian varieties $A$ such that $A[\ell]$ has a composition series with constituents in $\cE$?
\end{Ques}

Taking $\cE$ to be the set of all finite flat group schemes of order 2 over $\cO_K$ as constructed in \cite{TO} and $S=\emptyset$, we are led to the next definition.

\begin{Def}
Let $V$ be a finite flat group scheme over $\cO_K$.  We say that $V$ is {\em prosaic} if the field of points $K(V)$ is a 2-{\em extension} of $K$. Let $A$ be an abelian scheme over $\cO_K$. We say that $A$ is {\em prosaic} if $A[2]$ is prosaic, i.e., the 2-division field $K(A[2])$ is a 2-extension of $K$.
\end{Def}

For $V$ to be prosaic, it is equivalent that $V$ admit a filtration whose constituents are finite flat group schemes over $\cO_K$ of order 2.  By {\em group scheme}, we always mean {\em commutative finite flat group scheme}.

\medskip

Ramified primes over 2 introduce additional difficulties so, for the rest of this document, we reserve the symbol $K$ for a number field satisfying the following hypothesis.  

\vspace{5 pt}

\noindent {\bf Hypothesis K}:  \hspace{20 pt}  
\fbox{\parbox{190 pt}
{\begin{center} 
$K$ has odd discriminant, maximal order $\cO$ and trivial narrow class group.  
\end{center}}} 

\vspace{6 pt}

In \S \ref{Font}, we recall the bound of Abrashkin-Fontaine on higher ramification groups associated to group schemes over a local field.  It is used in Notation \ref{cFL1L2}, to  introduce an extension $\cF/K$ that contains the field of points $K(E)$ for every finite flat group scheme $E$ of exponent 2 over $\cO$.   If $\cF/K$ is a 2-extension, then every abelian scheme $A$ over $\cO$ is prosaic. 

For primes $\gp \vert 2\cO$, we define the {\em witness} field $K^\gp$ as follows.  Note that $K^\gp$ is a subfield of $\cF$ because it satisfies the ramification bound  \eqref{cFRamBd} that defines $\cF$.

\begin{Not}  \label{Kp}
Given a prime $\gp$ over 2 in $\cO$, set $\gc = 2 \gp^{-1}$ and let $K^\gp$ denote the maximal elementary $2$-extension of $K$ whose ray class conductor divides $\infty \gc^2$.
\end{Not}

We assume that one of the following conditions holds:  
\begin{equation} \label{D1D2Intro}
\fbox{\parbox{225 pt}
{\begin{tabular}{l l}
${\bfD}_1(\gp)\!:$ &$K^\gp=K$ for some $\gp$ dividing $2\cO.$ \\[3 pt]
$\bfD_2(\gp)\!:$   &$[K^\gp:K] =2$,  with $\gp$ inert in $K^\gp/K$ and\\[2 pt]
&${\bfD}_1(\gq)$ is false for all $\gq \vert 2\cO$.
\end{tabular}}}
\end{equation}
Equivalently, for some $\gp$, no quadratic extension of $K$ of ray class modulus $\infty\gc^2 $ splits over $\gp$.  In particular, $\cD_1(\gp)$ is automatic when $\gp = 2\cO$ is the unique prime over 2 in $\cO$.

Denote the maximal abelian extension of $K$ in $\cF$ by  $L_1$ and the maximal abelian extension of $L_1$ in $\cF$ by $L_2$.
\begin{center} 
\begin{tikzpicture}[scale=1]	
       \draw (0,0) node {$\bullet$};     \draw (.3,0) node {$K$};  
       \draw (0,.5) node {$\bullet$};   \draw (.35,.5) node {$K^\gp$};  
       \draw (0,1) node {$\bullet$};   \draw (.35,1) node {$L_1$};  
       \draw (0,1.5) node {$\bullet$}; \draw (.35,1.5) node {$L_2$};  
       \draw (0,0) -- (0,1.5);  
       \draw (-.65,1.35) node {abelian}; 
       \draw (-.2,1) -- (-.5,1) -- (-.5,0) -- (-.2,0);    \draw (-1.15,.55) node {abelian}; 
\end{tikzpicture}
\end{center}
One novelty of our paper is to use group theoretic implications of the bound on higher ramification defining $\cF$, rather than just the implied discriminant bound.  As an illustration, Propositions \ref{D1Type} and \ref{D2Type} describe the maximal Galois 2-extension of $K$ in $\cF$.

By Proposition \ref{Elem2}, the maximal abelian 2-extension of $K$ in $\cF$ is elementary abelian of rank $r_1+r_2$, where $r_1$ and $r_2$ are the number of real and complex places of $K$, respectively.  It becomes prohibitive to compute the field $L_1$ as $\fdeg{K}{\Q}$ increases.  Instead, we develop some group theory in \S \ref{GroupTheory} that will be used to test for abelian extensions of $L_1$ of odd degree by considering only quadratic extensions of $K$.  Similarly, abelian extensions of odd degree over $L_2$ can be detected from quadratic extension of $K^\gp$.

In \S \ref{GrpSchemes}, we consider group schemes $G$ over $\cO$ of order 2 and their extension classes by one another.  Let $A$ be a prosaic abelian scheme over $\cO$.  We give conditions for vanishing of some of those extension classes, under which there is a filtration of the form
$$
0 \subset V_1 \subset V_2 \subset V_3 \subset \dots \subset A[2^n], \quad n \ge 1,
$$
organized as follows.  For each $G$, there is a unique subquotient $V_{i+1}/V_i$ filtered entirely by copies of $G$.  For example, see Propositions \ref{MoveZMu} and \ref{sorter}.  We bound the exponent of each subquotient $V_{i+1}/V_i$, independently of $n$, in Corollaries \ref{z7Cor} and \ref{bdCor} and in Lemma \ref{ConstBd}.

Some applications of class field theory are contained in \S \ref{ClassField}.  Our algorithm uses Proposition \ref{NoThyProp} to test for finiteness of certain abelian 2-extensions of $K$ unramified outside 2 and infinity.    In Propositions \ref{DpTest1} and \ref{DpTest2}, we combine group theory from \S \ref{Font} and \S \ref{GroupTheory} to describe the maximal solvable extension of $K$ in $\cF$.  

We state conditions for the non-existence of prosaic abelian schemes over $K$ in Theorems \ref{2PrimeThm} and \ref{3PrimeThm}.  In \S \ref{HowDone}, we sketch an algorithm to sieve for fields $K$ that satisfy those conditions.   In the last step of the algorithm, Odlyzko bounds are used to force $\cF$ itself to be a 2-extension.  Then there are no abelian schemes over $\cO$ and so $K$ is a Fontaine field.  

With the help of \cite{LMF} and \cite{MAG} and only using GRH for number fields, we produced 26847 Fontaine fields, as detailed in Table \ref{punchline}. Two limitations are the lack of adequate tables of number fields of degree greater than 7 and the time involved in some of the tests.

While preparing this paper, we saw the preprint \cite{Tch}. Tchamitchian,  following \cite{Mes}, assumes all the standard conjectures on $L$-functions of abelian varieties and applies  Weil's explicit formulas to them. He exhibits hundreds of fields over which there are no everywhere good abelian varieties.    See Remark \ref{Tcha}.  

\medskip

\noindent {\bf Acknowledgments.} We owe a great debt to Ren\'e Schoof for his papers and in particular for making the preprint \cite{Sch7} available to us in July 2020.  We thank the referee for careful reading and for several helpful editorial comments.

\section{Applications of higher ramification bounds}  \label{Font}

We first establish notation and review some results in \cite{Ab,Fo}.  Consequences for number fields $K$ satisfying Hypothesis {\bf K} are in Propositions \ref{Elem2} and following.
 
Assume first that $F/E$ is a Galois extension of $p$-adic fields with Galois group $G$.  Using the lower numbering of \cite[Ch. IV]{Ser}, let $G_j$ run over the ramification subgroups of $G$ and let $\varphi_{F/E}$ be the Herbrand function: 
$$
\varphi_{F/E}(u) = \int_0^u \frac{dt}{\fdeg{G_0}{G_t}}.
$$
The upper numbering is given by $G^\alpha = G_j$ with $\alpha = \varphi_{F/E}(j)$.  Let
\begin{equation} \label{LastRam}
c_{F/E} = \max\{j \, \vert \, G_j \ne 1\} \quad \text{ and } \quad m_{F/E} = \varphi_{F/E}(c_{F/E}),
\end{equation}
with $c_{F/E} = m_{F/E} = -1$ when $F/E$ is unramified.  If $m = m_{F/E} \ge 0$, then $G^m \ne 1$ but $G^{m+\epsilon} = 1$ for all $\epsilon > 0$.  For a tower of fields $L \supseteq F \supseteq E$ with $L/E$ and $F/E$ Galois, we have $m_{F/E} \le m_{L/E}$ because of the surjection 
$$
\Gal(L/E)^\alpha \xrightarrow{\rm res} \Gal(F/E)^\alpha \quad \text{ for all } \alpha.
$$ 

Let $\gf = \gf(F/E)$ be the conductor exponent for an abelian extension $F/E$.  Equivalently, the ray class modulus for $F/E$ is $\gp^{\gf}$, where $\gp$ is the prime of $E$.  According to \cite[XV, \S2]{Ser},
$
\gf(F/E) =m_{F/E}+1,
$ 
with $\gf(F/E) = 0$ when $F/E$ is unramified.  Translation by an unramified extension of the base does not affect the conductor.   We state the bound on higher ramification using Serre's notation:   Fontaine's numbering is larger by 1.

\begin{theo}  \label{FontThm} \cite[Th\'{e}or\`{e}me A]{Fo}.  
Let $V$ be a finite flat group scheme of exponent $p^n$ over the ring of integers $\cO_{E}$ and let $e$ be the ramification index of $E/\Q_p$.  If $F$ is contained in the field of points $E(V)$, then
$$
m_{F/E} \le e \left(n+\frac{1}{p-1} \right)-1.
$$
\end{theo}

In our applications, $p = 2$, $e = 1$ $n = 1$, so $m_{F/E} \le 1$.  In particular, if $F/E$ is abelian then $\gf(F/E) \le 2$.

\begin{prop} \label{Tame}
Let $E/\Q_2$ be unramified and let $L/E$ be a Galois extension with Galois group $G = \Gal(L/E)$ and trivial tame ramification.  
\begin{enumerate}[{\rm i)}]
\item Then $m_{L/E} \le 1$ if and only if $G_2 = \{1\}$.  If so, the inertia group $\cI(L/E) = G_0 = G_1$ is an elementary $2$-group.   \vspace{4 pt}
\item Suppose that $L \supset F \supset E$, with $F/E$ also Galois.  Then $m_{L/E} \le 1$ if and only if both $m_{L/F} \le 1$ and $m_{F/E} \le 1$.
\end{enumerate}
\end{prop}

\begin{proof} \cite[Lemma 6]{BK1}. 
It suffices to assume that $L/E$ has some wild ramification. For item (i), suppose first that $m_{L/E} \le 1$.  Let $g_j$ be the order of the $j$-th non-trivial ramification group $G_j$ for $0 \le j \le c$, with $G_{c+1} = \{1\}$.  Since tame ramification is trivial, $g_0 = g_1$.  If $c \ge 2$, the Herbrand function gives a contradiction:
$$
m = \varphi_{L/E}(c) \ge \frac{1}{g_0}(g_1+g_2) = 1 + \frac{g_2}{g_1} > 1,  
$$ 
Thus $m = c = 1$ and $G_2 = \{1\}$.  Then $G_1$ is an elementary $2$-group, since that holds for successive quotients in the ramificaton filtration.  Suppose instead that $G_2 = \{1\}$.  Then $c_{L/E} = 1$, so $m_{L/E} = \varphi_{L/E}(1) = 1$.  Hence (i) is proved.

For item (ii), let $H = \Gal(L/F)$ and $\ov{G} = \Gal(F/E)$.  For $\alpha \ge 0$, let $\beta = \varphi_{L/E}(\alpha)$, so that $G^\beta = G_\alpha$.  Since the upper ramification numbering is compatible with quotients and $H_\alpha = H \cap G_\alpha$, the sequence 
$$
1 \to H_\alpha \to G^\beta \to \ov{G}^{\, \beta} \to 1
$$
is exact.  If $m_{L/E} \le 1$, then $G^{\beta} = \{1\}$ for all $\beta > 1$.  The same therefore holds for $\ov{G}^{\, \beta}$ and so $m_{F/E} \le 1$.  Also, $G_{2} = \{1\}$ by (i), so $H_2 = \{1\}$ and thus $m_{L/F} \le 1$.  Conversely, given $m_{F/E} \le 1$, we have $\ov{G}^{\, \beta} = \{1\}$ for all $\beta >1$.  Also given $m_{L/F} \le 1$, we have $H_2 = \{1\}$.  Hence $G^\beta = \{1\}$.
\end{proof}

\vspace{5 pt}

Next, we apply these local results to a number field $K$ which, by convention, always is assumed to satisfy Hypothesis {\bf K}.  If $F/K$ is a Galois extension, write $\cD_{\gP}(F/K)$ and $\cI_{\gP}(F/K)$, respectively, for the decomposition and inertia group at a prime $\gP$ of $F$.

\begin{Not} \label{cFL1L2}
The following notation is preserved throughout this paper:
\begin{enumerate}[\, $\bullet$]
\item The {\em Abrashkin-Fontaine extension} $\cF$ is the maximal extension of $K$ unramified at all finite places outside $2$ and satisfying 
\begin{equation} \label{cFRamBd}
\cD_{\gP}(\cF/K)^\beta = \{1\} \, \text{ for each prime } \gP \text{ over } 2 \text{ in } \cF \text{ and all }  \beta > 1.
\end{equation}

\vspace{2 pt}

\item Let $L_0 = K$.  For $j \ge 0$, let $L_{j+1}$ be the maximal abelian extension of $L_j$ in $\cF$, namely the fixed field of the commutator subgroup $\Gal(\cF/L_j)'$ of $\Gal(\cF/L_j)$.   
\end{enumerate}
\end{Not}

It follows from the ramification bound \eqref{cFRamBd} that an abelian extension $F$ of $K$ is contained in $\cF$ if and only if its ray class modulus divides $4\infty$ in $\cO$.  By Theorem \ref{FontThm}, if $V$ is a group scheme of exponent 2 over $\cO$, then the field of points $K(V)$ is contained in $\cF$.  In addition, if $V$ is prosaic, then $K(V)$ is contained in the maximal Galois 2-extension of $K$ in $\cF$.

\begin{Not}
For a Galois extension of number fields $F/E$ and a prime $\gP$ of $F$, let $m_{F/E}(\gP)$ be the upper number of the last non-trivial ramification group at $\gP$, as in \eqref{LastRam}.  Thus, the inertia group at $\gP$ satisfies $\cI_\gP(F/E)^\beta = \{1\}$ for all $\beta > m_{F/E}(\gP)$.  
\end{Not}  

\begin{prop} \label{Elem2}
The maximal abelian $2$-extension of $K$ inside $\cF$ is an elementary $2$-extension of rank $r_1+r_2$, where $r_1$ and $r_2$ are the number of real and complex places of $K$.
\end{prop}

\begin{proof}
Let $L$ be the maximal abelian $2$-extension of $K$ inside $\cF$.   Factor $2 \cO= \gp_1 \cdots \gp_n$ into distinct primes in $\cO$ and choose a prime $\gP_j$ of $L$ over each $\gp_j$.  Since $L/K$ is a $2$-extension, tame ramification at $\gP_j$ is trivial and so the inertia groups $\cI_j = \cI_{\gP_j}(L/K)$ are elementary 2-groups by Proposition \ref{Tame}(i).  By assumption, $L/K$ is abelian, so the subgroup $\cG$ of $\Gal(L/K)$ generated by all $\cI_j$ for $1 \le j \le n$ also is an elementary 2-group.   Suppose that $\cG$ is not all of $\Gal(L/K)$ and let $H$ be a subgroup of index 2 in $\Gal(L/K)$ containing $\cG$.  Then the fixed field $L^H$ is a quadratic extension of $K$, unramified at all $\gP_j$ and therefore unramified at all places over 2.  But $L$ already is unramified at other finite places and the narrow class group of $K$ is trivial, so we have a contradiction.  Hence $\Gal(L/K)$ is an elementary 2-group.

Because the class group of $K$ is trivial and $L/K$ does not ramify at odd places, each quadratic extension of $K$ inside $L$ has the form $K(\sqrt{u})/K$, where the Kummer generator $u$ is a unit outside primes over $2$.  To satisfy $m_{L/K}(\gP) \le 1$  for each prime $\gP$ over 2 in $L$, it is necessary and sufficient that, up to squares, $u$ also is a unit at $\gP$.  (For example, see \cite[Lemma C.6]{BK3}, verifying \cite[Ex.\! 3, p.\! 79]{Ser}.)  Hence we can take $u$ to range over the units of $K$.  The rank follows from the Dirichlet unit theorem, also including 2-power roots of unity. 
\end{proof}

\begin{Not}
Let $\cF_2$ be the maximal Galois 2-extension of $K$ in $\cF$.  
\end{Not}

In Lemma \ref{cF2Lemma} and Propositions \ref{D1Type} and \ref{D2Type} below, $\gp$ denotes a prime divisor of $2\cO$ and $\gP$ is a prime over $\gp$ in $\cF_2$.    Recall that $K^\gp$ is the witness field defined in Notation \ref{Kp}.  

\begin{lem} \label{cF2Lemma}
The inertia group  $\cI_\gP(\cF_2/K)$ is an elementary $2$-group.
\end{lem}

\begin{proof}
The ramification bound \eqref{cFRamBd} implies that $m_{\cF_2/K}(\gP) \le 1$ for every prime $\gP$ over 2 in $\cF_2$.  Since $\Gal(\cF_2/K)$ is a 2-group, tame ramification over $\gp$ is trivial, so $\cI_\gP(\cF_2/K)$ is elementary abelian by Proposition \ref{Tame}(i).  \qed
\end{proof}

\begin{prop}  \label{D1Type}
If $K^\gp = K$ for some $\gp$, then $\Gal(\cF_2/K)$ is an elementary $2$-group of rank $r_1+r_2$.  There is a unique prime $\gP$ over $\gp$ and it is totally ramified in $\cF_2/K$.
\end{prop}

\begin{proof} 
If $\cI_\gP = \cI_\gP(\cF_2/K)$ is not all of $G = \Gal(\cF_2/K)$, then it is contained in a maximal proper subgroup $H$ of $G$.  The fixed field of $H$ is unramified over $\gp$ and therefore satisfies the ray class conductor condition defining $K^\gp$ and so it equals $K$, contradicting $\fdeg{G}{H}=2$.  Hence $G = \cI_\gP$ and $\gP$ is the only prime over $\gp$ in $\cF_2$, totally ramified in $\cF_2/K$.   By Lemma \ref{cF2Lemma}, $G$ is an elementary 2-group.  The rank of $G$ is given by Proposition \ref{Elem2}.
\end{proof}  

\begin{prop}  \label{D2Type}
Assume that $\fdeg{K^\gp}{K}=2$, with $\gp$ inert in $K^\gp/K$ and let $P$ be the product of the distinct primes over $2$ in $K^{\gp}$.  The following properties hold.
\begin{enumerate}[{\rm i)}]
\item $G = \Gal(\cF_2/K)$ is the decomposition group $\cD_\gP = \cD_\gP(\cF_2/K)$ and the inertia group $\cI_\gP = \cI_\gP(\cF_2/K)$ is an elementary $2$-group with fixed field $K^\gp$.  \vspace{2 pt}

\item  $G$ is at most two-step nilpotent and $\cF_2$ is the maximal Galois $2$-extension of $K$ in $L_2$.  

\vspace{2 pt}

\item $\cF_2$ is the maximal abelian $2$-extension of $K^\gp$ with ray class modulus $\infty P^2$. 
\end{enumerate}
\end{prop}

\begin{proof}
If $\cD_\gP$ is a proper subgroup of $G$, then it is contained in a maximal proper subgroup $H$.   The fixed field $\cF_2^H$ is a quadratic extension of $K$ in which $\gp$ splits.  By \eqref{cFRamBd}, the ray class conductor of $\cF_2^H$ over $K$ divides $\infty \gc^2$, where $\gc = 2\gp^{-1}$ and so $\cF_2^H = K^\gp$.  Since $\gp$ is inert in $K^\gp/K$, we have a contradiction.  Hence $G = \cD_\gP$.

By Lemma \ref{cF2Lemma}, $\cI_\gP$ is an elementary 2-group.  The quotient $\cD_\gP/\cI_\gP$ is generated by the image of a Frobenius at $\gP$ and so the fixed field $F$ of $\cI_\gP$ is cyclic over $K$.   But then $\Gal(F/K)$ is an elementary 2-group by Proposition \ref{Elem2} and so $\fdeg{F}{K} \le 2$.  Since $\gp$ does not ramify in $K^{\gp}/K$, the field $K^\gp$ is contained in $F$.  Hence  $F = K^{\gp}$ and (i) is proved.
 
Since $\cD_\gP/\cI_\gP = \Gal(K^\gp/K)$ is abelian, the commutator $G'$ is contained in $\cI_\gP$ and therefore is elementary abelian.  The abelianization $G/G'$ is elementary abelian by Proposition \ref{Elem2}.  Hence $G$ is at most two-step nilpotent and so $\cF_2$ is contained in $L_2$.  Item (ii) now follows.  

By \eqref{cFRamBd}, $m_{\cF_2/K}(\gQ) \le 1$ for each prime $\gQ$ over $2$ in $\cF_2$.   Then Proposition \ref{Tame}(ii) implies that $m_{\cF_2/K^\gp}(\gQ) \le 1$.  Since $\cF_2/K^\gp$ is abelian, its ray class conductor divides $\infty P^2$.  Hence $\cF_2$ is contained in the maximal abelian $2$-extension $M$ of $K^\gp$ with that conductor.  Conversely, the given conductor implies that $m_\gq(M/K^\gp) \le 1$ for each prime $\gq$ over 2 in $M$.  It follows from Proposition \ref{Tame}(ii) that $m_\gq(M/K) \le 1$.  Hence $M$ is contained in $\cF_2$, so $M = \cF_2$ and (iii) holds.
\end{proof}

\section{Some group theory}  \label{GroupTheory}
Here we provide group theoretic tools to study the maximal solvable extension of $K$ inside $\cF$.  We write $G'$ for the commutator subgroup of a group $G$ and $G'' = (G')'$.  Then $G^{ab} = G/G'$ is the abelianization of $G$.

\begin{prop} \label{OddDihedral}
Let $\Gamma$ be a finite group with the following properties.
\begin{enumerate}[{\rm \, i)}]
\item $\Gamma'$ is abelian and an odd prime $p$ divides $\vv{\Gamma'}$. \vspace{2 pt}

\item $\Gamma^{ab}$ is an elementary $2$-group.  
\end{enumerate}
Then $\Gamma$ has a quotient isomorphic to the dihedral group $D_p$.
\end{prop}

\begin{proof}
Replace $\Gamma$ by $G = \Gamma/(\Gamma')^p$.  Then $G' = \Gamma'/(\Gamma')^p$ is an elementary abelian $p$-group and $G^{ab} \simeq \Gamma^{ab}$ is an elementary 2-group.  It suffices to show that $G$ has a $D_p$-quotient.  

If $H$ is any 2-Sylow subgroup of $G$, then $G = G' \rtimes H$ is a semidirect product of $H$ by the normal subgroup $G'$.  As an $\F_p[H]$-module, $G' = \oplus W_j$ is a direct sum of 1-dimensional eigenspaces for $H$.   Since $G$ is not abelian, there is at least one such eigenspace, say $W_1 = \lr{b} \simeq \F_p$ on which some involution $a$ in $H$ acts by inversion.  Since $G$ acts on $G'$ via $G/G' \simeq H$, the centralizer $H_0 = \{h \in H \mid hbh^{-1} = b \}$ of $b$ in $H$ is a normal subgroup of $G$ and $H/H_0 \simeq \lr{a}$.  Also, $V = \oplus_{j \ge 2} W_j$ is a normal subgroup of $G$ with $G'/V \simeq \lr{b}$ and thus $G/H_0 V \simeq \lr{a,b} \simeq D_p$.
\end{proof}

\begin{prop} \label{My9Lem}
Let $\Gamma$ be a finite group with the following properties.
\begin{enumerate}[{\rm i)}]
\item $\Gamma''$ is abelian and $\vv{\Gamma''}$ is divisible by an odd prime $p$.  \vspace{2 pt}
\item $G = \Gamma/\Gamma''$ is a $2$-group with a subgroup $H$ of index $2$ and exponent $2$.
\end{enumerate}
Let $\Gamma_0$ be the inverse image of $H$ under the projection map $\Gamma \to G$.  Then $\Gamma_0$ has a $D_p$-quotient.
\end{prop}

\begin{proof}

Since $\Gamma''$ projects trivially to $G$, it is contained in $\Gamma_0$ and $\Gamma_0/\Gamma''$ is isomorphic to $H$.  But $H$ is abelian, so $\Gamma_0'$ is contained in $\Gamma''$.  For the reverse inclusion, note that $\Gamma_0$ is normal in $\Gamma$ and 
$$
\Gamma/\Gamma_0 \simeq (\Gamma/\Gamma'')/(\Gamma_0/\Gamma'') \simeq G/H \simeq \Z/2\Z.
$$
Since this quotient is abelian, $\Gamma'$ is contained in $\Gamma_0$.  Hence $\Gamma''$ is contained in and therefore equal to $\Gamma_0'$.  Thus $\Gamma_0'$ is abelian and $p$ divides $\vv{\Gamma_0'}$.  In addition 
$$
\Gamma_0^{ab} = \Gamma_0/\Gamma_0' = \Gamma_0/\Gamma'' \simeq H
$$
is an elementary 2-group.  Thus $\Gamma_0$ has a $D_p$-quotient by Proposition \ref{OddDihedral}.
\end{proof}

\begin{prop} \label{Wreath}
Let $G$ be a non-abelian $2$-group.  The following are equivalent:
\begin{enumerate}[{\rm i)}]
\item $G$ contains a subgroup $H$ of index $2$ and exponent $2$.   \vspace{2 pt}
\item $G \simeq H_1 \times W$, where $H_1 \simeq (\Z/2\Z)^e$ for some $e \ge 0$ and $W$ is the wreath product $(\Z/2\Z)^s \wr (\Z/2\Z)$ with $s \ge 1$.
\end{enumerate}
Under these conditions, $G$ has a dihedral quotient $D_4\simeq(\Z/2\Z) \wr (\Z/2\Z)$.
\end{prop}

\begin{proof}
Assume (i).  Since $G/H \simeq \Z/2\Z$ is abelian, $G'$ is contained in $H$ and so $G'$ is elementary abelian.  Let $s \ge 1$ be the rank $G'$.  Because $G$ is not abelian, it contains an element $b$ of order 4.  Then $b^2$ is in $H$ and $G = \lr{H,b}$.  The quotient $G/H$ acts on $H$ via conjugation by $b$.   As an $\F_2$-vector space for this action, $H$ decomposes into Jordon blocks (written additively):
$$
H = \left(\oplus_{j=1}^{d} \, Y_j \right) \oplus \left(\oplus_{i=1}^{e} \, X_i \right)
$$
where $\dim Y_j = 2$, $\dim X_i = 1$ and so $\dim H = 2d+e$.  

For $1 \le j \le d$, we can find $a_j$ such that $Y_j = \F_2[G/H] \, a_j$ and so $Y_j = \lr{a_j,a_{d+j}}$ with $a_{d+j} = ba_jb^{-1}$.  Let $A = \lr{a_1, \dots, a_d}$, $\widetilde{A} = \lr{a_{d+1}, \dots, a_{2d}}$ and $H_1 = \oplus_{i=1}^{e} \, X_i$.  Then $H$ is an internal direct product
$$
H = H_1 \times A \times \widetilde{A} \, \text{ and } \, G = \lr{H_1,A,b}.
$$
It follows that an $\F_2$-basis for $G'$ is given by the commutators $[a_j,b] = a_j a_{d+j}$   Hence $d = s = \rk G'$ and $\rk H_1 = e$.  

The action of $G/H$ interchanges $A$ and $\widetilde{A}$.  If $W = \lr{A,b}$, then $W$ is a normal subgroup of $G$ isomorphic to $(\Z/2\Z)^s\!\wr\! (\Z/2\Z)$ and $\vv{W} = 2\vv{A}^2 = 2^{2s+1}$.   Also, $G = \lr{H_1,A,b} = H_1 W$.  This is an internal direct product because 
$$
\begin{array}{rcl}
2^{2s+e+1} &=& 2 \vv{H} =  \vv{G} = \vv{H_1W} = \vv{H_1} \vv{W}/\vv{H_1 \cap W} \\[4 pt]
&=& 2^{2s+e+1}/\vv{H_1 \cap W}
\end{array}
$$
and so $H_1 \cap W = \{1\}$.  It follows that (i) implies (ii).

Assume (ii).  By construction of the wreath product, there are subgroups $A$ and $\widetilde{A}$ of $W$ isomorphic to $(\Z/2\Z)^s$ and interchanged by the action of $\Z/2\Z$.  Then $A \times \widetilde{A}$ is an elementary $2$-group of index 2 in $W$ and so $H = H_1 \times A \times \widetilde{A}$ satisfies (i).  

If $B = \lr{a_2, \dots, a_d}$, then $G/(H_1 \times B \times bBb^{-1})$ is isomorphic to $D_4$.
\end{proof}

The next result of Schoof is useful for limiting odd primes in $\fdeg{\cF}{K}$.

\begin{lem}\cite[Corollary 3.2]{Sch3}.    \label{9lem}
Let $\Gamma$ be a finite group such that $\Gamma'/\Gamma''$ is a $2$-group.  If $\vv{\Gamma''} < 25$, then either $\Gamma''$ is a $2$-group or $9$ divides $\vv{\Gamma''}$.
\end{lem}

\numberwithin{equation}{section}

\section{The relevant group schemes} \label{GrpSchemes}

We quote some general tools from Schoof's work.  Then we use this information to address the extension problem for group schemes of order 2 that might occur in the kernel of multiplication by $2$ on the abelian varieties we consider.  By {\em group scheme} we always mean {\em commutative finite flat group scheme}.  

\begin{lem} \cite[Prop. 2.3]{Sch3}. \label{Artin1} 
Let $R$ be a commutative Noetherian ring, $p$ an element of $R$ and $\widehat{R}=\underleftarrow{\lim}\, R/p^nR$.   Let $\un{Gr}_R$ be the category of $R$-group schemes and let $\un{C}$ be the category of triples $(G_1,G_2,\theta)$, where $G_1$ is an $\widehat{R}$-group scheme, $G_2$ is an $R[\frac{1}{p}]$-group scheme and
$$
\theta\!: \, G_1 \, \otimes_{\widehat{R}} \, \widehat{R}[\textstyle{\frac{1}{p}]}  \, \longrightarrow  \,  G_2  \, \otimes_{R[\textstyle{\frac{1}{p}}]} \, \widehat{R}[\textstyle{\frac{1}{p}]} 
$$
is an isomorphism of $\widehat{R}[\frac{1}{p}]$-group schemes.  Morphisms in $\un{C}$ are pairs of morphisms of $G_1$ and $G_2$ that are compatible with $\theta$.  There is an equivalence of categories $\underline{Gr}_R \to \underline{C}$  given by the functor that sends an $R$-group scheme $G$ to the triple 
$$
(G\otimes_R\widehat{R},\, G\otimes_R R[\textstyle{\frac{1}{p}}], \, \text{id}\otimes_R \widehat{R}[\textstyle{\frac{1}{p}}]).
$$ 
\end{lem} 

\begin{cor} \cite[Cor. 2.4]{Sch3}. 
Let $G$ and $H$ be  $R$-group schemes.  There is a natural exact {\em Mayer-Vietoris}   sequence
\begin{equation} \label{MV}  
\begin{array}{l} 
  \Hom_{\widehat{R}[\frac{1}{p}]}(G,H)  \leftarrow \Hom_{\widehat{R}}(G,H) \times  \Hom_{{R[\frac{1}{p}]}}(G,H) \leftarrow \Hom_R(G,H) \leftarrow  0  \vspace{2 pt} \\
 \hspace{15 pt}     \delta \, \downarrow    \vspace{2 pt} \\
   \Ext^1_{R}(G,H)   \rightarrow  \Ext^1_{\widehat{R}}(G,H) \times\Ext^1_{{R[\frac{1}{p}]}}(G,H) \rightarrow   \Ext^1_{\widehat{R}[\frac{1}{p}]}(G,H).
\end{array}
\end{equation}
The map $\delta$ is described in the reference.
\end{cor}

\begin{lem} \cite[Prop. 4.2]{Sch7a}. \label{Kum} 
If $F$ is a number field with ring of integers $R$ and trivial narrow class group, then there is an exact sequence of Kummer theory
$$
1\to\{\pm 1\} \to \Ext^1_{R}(\Z/2\Z,\Mu_2)\to R^\times\!/R^{\times 2}\to 1.
$$
The natural map $ \Ext^1_{R}(\Z/2\Z,\Mu_2)\to \Ext^1_{F}(\Z/2\Z,\Mu_2)$ is injective and  any extension $0\to \Mu_2 \to G\to \Z/2\Z\to 0$ is determined by the Galois module $G(\ov{F})$. 
\end{lem}

\begin{lem} \cite[Lemma 2.8]{CG}.
Let $R$ be a Dedekind domain with field of fractions $F$ of characteristic
$0$, and let $G$ be a finite flat group scheme over $R$. Then there is a bijective correspondence associating each closed finite flat subgroup scheme $H \subseteq G$ over $R$ to the $\Gal(\ov{F}/F)$-submodule $H(\ov{F}) \subseteq G(\ov{F})$.
\end{lem}

\begin{Not} \label{Grho}
We continue to assume that $\cO$ is the ring of integers in a number field $K$ satisfying Hypothesis {\bf K}.  Let $\rho$ be a divisor of 2 in $\cO$ and write $G_\rho$ for the group scheme of order 2 over $\cO$ with Hopf algebra
$
\cO[X]/(X^2-\rho X),
$
as in \cite{TO}.
\end{Not}

Recall that $G_{\rho_1}$ and $G_{\rho_2}$ are isomorphic if and only if $\rho_1=\rho_2 u$ for some unit $u$ in $\cO$.  If so, we say that $\rho_1$ and $\rho_2$ are associates.  The Cartier dual group scheme $G_\rho^\wedge$ is given by $G_{\bar{\rho}}$, where $\rho \bar{\rho}=2$.  The only \'{e}tale group scheme of order 2 over $\cO$ is $\Z/2\Z$, equal to $G_u$ for any unit $u$.  Over the completion $\cO_\pi$ at a prime $\pi \vert 2$, we have
\begin{equation} \label{SimpleLocal}
G_\rho \simeq \begin{cases} \Mu_2 &\text{if } \pi \mid \rho, \\ \Z/2\Z  &\text{if } \pi \nmid \rho. \end{cases} 
\end{equation}
Note that group schemes of 2-power order over $\cO[\frac{1}{2}]$ can be identified with their Galois modules over $K$.

\vspace{5 pt}

The next result generalizes Proposition 4.3(ii) of  \cite{Sch7a}.  

\begin{prop}  \label{ExtDim}
Let $\rho$ and $\rho'$ be divisors of $2$ in $\cO$ and let
$$
S = \{{\rm primes} \hspace{3 pt} \pi \hspace{2 pt} : \hspace{2 pt} \pi \vert \rho, \hspace{2 pt} \pi \! \nmid \! \rho'\}, \hspace{3 pt} S' = \{{\rm primes} \hspace{3 pt} \pi' \hspace{2 pt} : \hspace{2 pt} \pi' \vert \rho', \hspace{2 pt} \pi' \! \nmid \! \rho\},  \hspace{3 pt} P' = \textstyle{\prod_{\pi' \in S'}} \, \pi'.
$$
Assume that $s = \vv{S}  \ge 1$ and let $M$ be the maximal elementary $2$-extension of $K$ of conductor $\infty (P')^2$, split over all $\pi$ in $S$.  If $G$ is an extension of the form 
\begin{equation} \label{Ext21}
0 \to G_{\rho'} \to G \to G_\rho \to 0,
\end{equation}
then $2G = 0$, the field of points $K(G)$ is quadratic over $K$ and is contained in $M$.  Also, $\vv{\Ext_{\cO}^1(G_\rho,G_{\rho'})}  = 2^{s-1} \, \fdeg{M}{K}$.
\end{prop}

\proof 
By assumption, $S$ is not empty and $K(G)/K$ is unramified at odd places.  At each prime $\pi$ in $S$, the extension \eqref{Ext21} splits over $\cO_\pi$ and $\Hom_{\cO_\pi}(G_\rho,G_{\rho'})=0$ because of the connected-\'etale exact sequence.  Then $2 G = 0$ over $\cO_\pi$ and therefore over $\cO$.  Hence $K(G)$ is at most quadratic over $K$.

At primes $\lambda$ over 2 dividing neither $\rho$ nor $\rho'$, the extension $G$ is \'etale, so $K_\lambda(G)/K$ is unramified.  By duality, $K_\lambda(G)/K$ is unramified if $\lambda$ divides both $\rho$ and $\rho'$.  It follows that only the primes $\pi'$ in $S'$ may ramify in $K(G)$.   For those, the conductor exponent satisfies $\gf_{\pi'} \le 2$ by Theorem \ref{FontThm}.  Thus $K(G)$ is contained in $M$ and is quadratic over $K$.  In addition, the image of the Galois module extension class $[G]$ in $\Ext^1_K(G_\rho,G_{\rho'})$ corresponds to a character $\chi_G\!: \, \Gal(M/K) \to \Mu_2$.   

We summarize information about homomorphisms needed to apply the Mayer-Vietoris sequence \eqref{MV}.  Let $\widehat{\cO} = \cO \otimes \Z_2 = \oplus_{\lambda \vert 2} \, \cO_\lambda$, let $\widehat{K} = K \otimes \Z_2$ and let $t$ be the number of primes over 2 in $\cO$.  
 \begin{enumerate}[\, i)]
 \item $\Hom_{\cO_\lambda}(\Z/2\Z,\Mu_2) = \F_2$ and $\Hom_{\cO_\lambda}(\Mu_2,\Z/2\Z) = 0$.  \vspace{2 pt}
 
\item $\Hom_{\cO}(G_\rho,G_{\rho'})=0$, since these homomorphisms are trivial at all $\pi$ in $S$.  \vspace{2 pt}

\item $\Hom_{\widehat{\cO}}(G_\rho,G_{\rho'}) = \oplus_{\lambda \vert 2} \Hom_{\cO_\lambda}(G_\rho,G_{\rho'}) \simeq \F_2^{t-s}$.  \vspace{2 pt}

\item $\Hom_K(G_\rho,G_{\rho'}) \simeq \F_2$, since $G_\rho$ and $G_{\rho'}$ become trivial Galois module of order 2 over $K$.   \vspace{2 pt}

\item $\Hom_{\widehat{K}}(G_\rho,G_{\rho'}) = \oplus_{\lambda \vert 2}\Hom_{K_\lambda}(G_\rho,G_{\rho'}) \simeq \F_2^t$, since $G_\rho$ and $G_{\rho'}$ are trivial Galois modules for the decomposition group at each prime $\lambda$ over $2$.
\end{enumerate}
Then the Mayer-Vietoris sequence \eqref{MV} for $\Ext_{\cO}^1(G_\rho,G_{\rho'})$ reduces to 
$$
0 \to 0 \to  \F_2^{t-s} \times  \F_2 \to \F_2^t \to \Ext_{\cO}^1(G_\rho,G_{\rho'}) \xrightarrow{f} \dots
$$ 
Next, we analyze the image of $f$ in the remaining fragment of the sequence:
$$
\Ext_{\cO}^1(G_\rho,G_{\rho'}) \xrightarrow{f} \Ext_{\widehat{\cO}}^1(G_\rho,G_{\rho'}) \times \Ext_K^1(G_\rho,G_{\rho'}) \to \Ext_{\widehat{K}}^1(G_\rho,G_{\rho'}).
$$ 
Let $[\cH_\lambda]$ denote an extension class in $\Ext_{\cO_\lambda}^1(G_\rho,G_{\rho'})$.  Let $[\cG]$ denote an extension class in $\Ext_K^1(G_\rho,G_{\rho'})$.   Then $(([\cH_\lambda])_\lambda, \cG(\ov{K})) = f(\cG)$ is in the image of $f$ if and only if, for all $\lambda$, we can find $[\cH_\lambda]$ whose image in in $ \Ext_{K_\lambda}^1(G_\rho,G_{\rho'})$ agrees with the images of $[\cG]$ as extension classes of Galois modules over $K_\lambda$. For this, we arrive at precisely the conductor and splitting conditions defining $M$.  Thus the image of $f$ corresponds to the group of characters $\chi\!: \, \Gal(M/K) \to \Mu_2$.  It follows that 
$
\vv{\Ext_{\cO}^1(G_\rho,G_{\rho'})}= 2^{s-1} \, \fdeg{M}{K}.
$
\qed

\vspace{8 pt}

Special case (i) below is in \cite[Prop. 4.3(ii)]{Sch7a}.  For the explicit description of $\Ext_{\cO}^1(G_\rho,\Z/2\Z)$, see the discussion after the proof of that Proposition.

\begin{cor}  \label{ExtDimCor}
If $s = \#\{\pi \text{ dividing } \rho \, \} \ge 1$, then:
\begin{enumerate}[{\rm i)}] 
\item $\Ext_{\cO}^1(G_\rho,\Z/2\Z)$ is an elementary $2$-group of rank $s-1$.  \vspace{2 pt}
\item Each class $[G]$ in $\Ext_{\cO}^1(G_\rho,\Z/2\Z)$ corresponds to a factorization $\rho=\rho_1\rho_2$ such that $G \simeq G_{\rho_1} \times G_{\rho_2}$ and \,
$
 0 \to \Z/2\Z \to G_{\rho_1} \times G_{\rho_2} \to G_\rho \to 0.
$
\end{enumerate}
\end{cor}

\begin{prop}  \label{AltNonSplit}  
Let $2 = \pi_1 \pi_2 \pi_3$ be the product of three primes in $\cO$.  Suppose that there are no quadratic extensions of $K$ with conductor dividing $\infty\pi_3^2$, split over $\pi_1$ and $\pi_2$.   Then $\vv{\Ext^1_\cO(G_{\pi_1 \pi_2},G_{\pi_3})} = 2$ and if $G$ represents the non-trivial class, then $K(G)=K$.  Furthermore, $G$ also represents a  non-trivial class in $\Ext^1_\cO(G_{\pi_j \pi_k},G_{\pi_i})$ for each ordering of $\{i,j,k\} = \{1,2,3\}$.
\end{prop}

\begin{proof}
The class of $G$ in $\Ext^1_\cO(G_{\pi_1 \pi_2},G_{\pi_3})$ is represented by an exact sequence
\begin{equation} \label{312}
0 \to G_{\pi_3} \to G \to G_{\pi_1 \pi_2} \to 0.
\end{equation} 
Apply Proposition \ref{ExtDim} with $P' = \pi_3$ and $S =\{\pi_1,\pi_2\}$.  Then $2G = 0$ and our assumption on quadratic extensions of $K$ implies that  $K(G) = K$.  Hence $\vv{\Ext^1_\cO(G_{\pi_1 \pi_2},G_{\pi_3})} = 2$ and the Galois module $G(K)$ splits as an extension of $\F_2$ by $\F_2$.

According to \eqref{312}, there is a closed subgroup scheme $H_1$ of $G$ isomorphic to $G_{\pi_3}$.  We wish to identify the closed subgroup schemes of $G$ associated to the other points of order 2 in $G(K)$.  If $H_2 \simeq G_\rho$ is any such, then the Galois modules of $H_1$ and $H_2$ generate $G(K)$.

If $\pi_3 \vert \rho$, then the connected component of $G$ at $\pi_3$ is 2-dimensional and we have a contradiction.  Hence, up to a unit, $\rho$ is in the set $\{1, \pi_1, \pi_2, \pi_1 \pi_2 \}$.  But $\rho$ and $\pi_1 \pi_2$ are not associates because a closed subgroup scheme of $G$ isomorphic to $G_{\pi_1\pi_2}$ would split exact sequence \eqref{312}.

If $\rho$ is a unit, then $H_2 \simeq \Z/2\Z$ leads to an exact sequence of the form
$$
0 \to \Z/2\Z \to G \to \Mu_2 \to 0,
$$
in which the quotient is isomorphic to $\Mu_2$ because $G$ has a non-trivial connected component at each of the primes $\pi_1,\pi_2,\pi_3$.  By Corollary \ref{ExtDimCor}(ii), $G$ is isomorphic to a product $G_{\rho_1} \times G_{\rho_2}$ with $\rho_1 \, \rho_2 = 2$.  One of these factors, say $\rho_1$ is divisible by $\pi_3$.  Note that $a = \rho_1/\pi_3$ cannot be a unit, or $G_{\rho_2}$ would be a closed subgroup scheme of $G$ isomorphic to $G_{\pi_1 \pi_2}$ which has already been ruled out.  Hence $a$ is divisible by an additional prime factor of 2.   But then the connected component of $G$ at $\pi_3$ is 2-dimensional, as in the contradiction noted above.  It follows that up to a unit, $\rho$ is $\pi_1$ or $\pi_2$ and we shall see that both occur.

Suppose that $H_2 \simeq G_{\pi_1}$.  Then there is an exact sequence
\begin{equation} \label{123}
0 \to G_{\pi_1} \to G \to G_{\pi_2 \pi_3} \to 0, 
\end{equation}
in which the quotient is isomorphic to $G_{\pi_2 \pi_3}$ because $G$ has a non-trivial connected component at each of the primes $\pi_2$,  $\pi_3$.  The sequence \eqref{123} does not split because $G$ has no closed subgroup scheme isomorphic to $G_{\pi_2 \pi_3}$.  Thus $G$ represents a non-trivial class in $\Ext^1_\cO(G_{\pi_2 \pi_3}, \pi_1)$.

Let $P$ and $Q$ be generators for the Galois submodules of $G(K)$ associated to $H_1$ and $H_2$, respectively.  The closed subgroup scheme $H_3$ of $G$ associated to the Galois module $\lr{P+Q}$ is isomorphic to $G_\tau$ for some $\tau$ in $\{\pi_1,\pi_2\}$.  If $\tau$ and $\pi_1$ are associates, then $H_2$ and $H_3$ create a 2-dimensional connected component of $G$ at $\pi_1$ and so a contradiction.  Hence $\tau$ and $\pi_2$ are associates and $G$ also represents a non-trivlal class in $\Ext^1_\cO(G_{\pi_1 \pi_3}, \pi_2)$, as argued above.
\end{proof}

\begin{prop}  \label{z7} 
Let $2 = \rho \bar{\rho}$ in $\cO$, with neither $\rho$ nor $\bar{\rho}$ a unit. If there is no quadratic extension of $K$ whose conductor divides $\infty\rho^2$, then 
$$
\Ext^1_\cO(G_\rho,G_\rho) = \Ext^1_\cO(G_{\bar{\rho}},G_{\bar{\rho}}) =  0.
$$
If $V$ is a group scheme over $\cO$ filtered by copies of $G_\rho$ or filtered by copies of $G_{\bar{\rho}}$, then $2V = 0$.                                                                                                                                                                         
\end{prop}

\begin{proof}
This argument parallels the proof in \cite[Prop. 4.2(iv)]{Sch3}.   Consider a non-trivial extension $0 \to G_\rho \to G \to G_\rho \to 0$.  Since the points of $G_\rho$ are in $K$, the degree $\fdeg{K(G)}{K} \le 2$.  Let $\pi$ range over the primes of $\cO$ above 2.  By \eqref{SimpleLocal}, $G$ becomes an extension of $\Mu_2$ by $\Mu_2$ over $\cO_\pi$ if $\pi\vert \rho$.  Also, $G$ becomes an extension of $\Z/2\Z$ by $\Z/2\Z$ over $\cO_\pi$ if $\pi \nmid \rho$.

Let $t$ be the number of primes over 2 in $\cO$ and recall that $\widehat{\cO} = \cO \otimes \Z_2 = \oplus_{\pi \vert 2} \, \cO_\pi$.   Based on \eqref{SimpleLocal}, we have
$\Hom_{\cO}(G_\rho,G_\rho) = \F_2$ and 
$$
\Hom_{\widehat{\cO}}(G_\rho,G_\rho) = \oplus_{\pi \vert 2} \Hom_{\cO_\pi}(G_\rho,G_\rho) = \F_2^t.
$$
As a result, the Mayer-Vietoris sequence \eqref{MV} for $\Ext_{\cO}^1(G_\rho,G_\rho)$ reduces to 
$$
0 \to \F_2 \to \F_2^t \times \F_2 \to  \F_2^t \to \Ext_{\cO}^1(G_\rho,G_\rho) \xrightarrow{f} \dots
$$ 
Therefore, $f$ is injective in the continuation of the above sequence:
$$
\Ext_{\cO}^1(G_\rho,G_\rho) \xrightarrow{f} \Ext_{\widehat{\cO}}^1(G_\rho,G_\rho) \times \Ext_K^1(G_\rho,G_\rho).
$$

\vspace{2 pt}

$\bullet$ Case 1:  Suppose that $G$ has exponent 2 over $\cO_{\pi_0}$ for some $\pi_0 \vert 2$.  Then $2G = 0$ globally over $\cO$ and so for all $\pi \vert 2$.   Because $G$ is \'etale at primes $\pi \nmid \rho$, the extension $K_\pi(G)/K_\pi$ is unramified and at most quadratic.   By duality, the same is true for $K_\pi(G)/K_\pi$ when $\pi \vert \rho$.  But then $K(G)$ is an unramified extension of $K$ at all finite places and so $K(G) = K$.  It follows that the Galois module $G(K)$ splits as an extension of $\F_2$ by $\F_2$ of exponent 2.  

We have already seen that the second coordinate of $f([G])$ is the trivial extension class over $K$.  Over the completion $\cO_\pi$, $G$ is an extension of $\Mu_2$ by $\Mu_2$ or $\Z/2\Z$ by $\Z/2\Z$.  Such an extension is determined by its Galois module \cite{Ray}, which is a split extension in our case.  Hence $f([G])$ is trivial and so $G$ splits over $\cO$.

\vspace{4 pt}

$\bullet$ Case 2:  Suppose that $G$ has exponent 4 over $\cO_\pi$ for all $\pi \vert 2$.  Let $\chi$ be the character for the action of Galois on the points of $G$ and let $N_\chi$ be the finite part of the conductor of $\chi$.  Consider the dual extension 
$$
0 \to G_{\bar{\rho}}\to G^\wedge \to G_{\bar{\rho}} \to 0, 
$$ 
with associated character $\chi'$ and conductor $N_{\chi'}$.  By duality, $\chi\chi' = w_4$, where $w_4$ is the character of $\Gal(K(i)/K)$.  Since 2 is unramified in $K/\Q$, the ray class conductor for $K(i)/K$ is $4 \infty$ and so $N_\chi \, N_{\chi'} = 4$.  But $K(G_\rho)/K$ is unramified when $\pi \nmid \rho$, so $N_\chi = \rho^2$ and similarly, $N_{\chi'} = \bar{\rho}^{\, 2}$, neither of which is a unit.  By assumption, there is no quadratic extension of $K$ of conductor dividing $\rho^2$, so we have a contradiction.

Finally, the splittings just verified imply that $2V = 0$.
\end{proof}

\begin{cor} \label{z7Cor}
Suppose that $K^{\gp} = K$ for some prime divisor $\gp$ of $2\cO$ and let $\rho \vert 2$ in $\cO$, with neither $\rho$ nor $2/\rho$ a unit.  If $V$ is filtered by copies of $G_\rho$, then $2V = 0$.
\end{cor}

\begin{proof}
By assumption, there is no quadratic extension of $K$ whose conductor divides $\infty \gc^2$, where $\gc = 2 \gp^{-1}$.  Then $\rho \vert \gc$ or else $\bar{\rho} \vert \gc$ and so $2V = 0$ by Proposition \ref{z7}.
\end{proof}

\begin{Def} \label{Fils}
If $\gF\!:  0 = H_0 \subset H_1 \subset \dots H_{s-1} \subset H_s = V$ is a filtration of a finite flat group scheme $V$ by closed flat subgroup schemes, then its grading is the list of consecutive quotients 
$$
\grad(\gF) = [H_1/H_0, H_2/H_1, \dots H_s/H_{s-1}].
$$  
Suppose that some subquotient $X = H_{i+2}/H_i$ is a split extension of group schemes of order 2: 
$$
0 \to G_{\rho'} \to X \to G_\rho \to 0,
$$
where $\rho$ and $\rho'$ and not associates,  $H_{i+1}/H_i \simeq G_{\rho'}$ and $H_{i+2}/H_{i+1} \simeq G_\rho$.  Then there also is a filtration
$$
\gF{\, '}\!: \,  0 = H_0 \subset \dots H_{i} \subset H'_{i+1} \subset H_{i+2}  \subset \dots \subset H_s = V
$$
with $H'_{i+1}/H_i \simeq G_{\rho}$ and $H_{i+2}/H'_{i+1} \simeq G_{\rho'}$.  We say that $G_\rho$ {\em moves to the left} and $G_{\rho'}$ {\em moves to the right}. 
\end{Def}

\begin{prop}  \label{MoveZMu} 
If $V$ is a prosaic group scheme over $\cO$, then there is a filtration 
\begin{equation} \label{ZMu1}
0 \subseteq V_1 \subseteq V_2 \subseteq V
\end{equation}
with composition series that have the following properties.
\begin{enumerate}[{\, \rm i)}]
\item All constituents of $V_1$ are isomorphic to $\Mu_2$.  \vspace{2 pt}
\item All constituents of $V/V_2$ are isomorphic to $\Z/2\Z$. \vspace{2 pt}
\item  No constituent of $V_2/V_1$ is isomorphic to $\Mu_2$ or $\Z/2\Z$.
\end{enumerate}
\end{prop}

\begin{proof} 
Choose a composition series $\gF$ for $V$ with the minimal number of constituents isomorphic to $\Mu_2$, if any.  If there is an obstacle to moving $\Mu_2$ to the left, then there is a subquotient $X =  H_{i+2}/H_i$ in the filtration of Definition \ref{Fils} such that
\begin{equation} \label{IsoX}
0 \to G_{\rho} \to X \to \Mu_2 \to 0 
\end{equation}
does not split.

According to the dual of Corollary \ref{ExtDimCor}(ii), $X$ is isomorphic to $G_{\rho_1} \times G_{\rho_2}$ for some factorization $\rho_1 \rho_2 = 2 \rho$ in $\cO$.  Since \eqref{IsoX} does not split, neither $G_{\rho_1}$ nor $G_{\rho_2}$ is isomorphic to $\Mu_2$. If the original composition series is modified to reflect this product, then there are fewer constituents isomorphic to $\Mu_2$ and so, a contradiction.  Thus, there is a closed subgroup scheme $V_1$ of $V$ satisfying item (i), with no constituent of the induced filtration for $W = V/V_1$ isomorphic to $\Mu_2$.

Similar reasoning applies to $W$.  Among the composition series for $W$ with no constituent isomorphic to $\Mu_2$, choose one with the minimal number of constituents that are isomorphic to $\Z/2\Z$.  If $\Z/2\Z$ cannot be moved to the right, past $G_\rho$ with neither $\rho$ nor $2/\rho$ a unit, then the filtration has a subquotient $X'$ that is a non-split extension 
\begin{equation} \label{IsoX'}
0 \to \Z/2\Z \to X' \to G_{\rho} \to 0.
\end{equation}
By Corollary \ref{ExtDimCor}(i), $X' \simeq G_{\rho_1} \times G_{\rho_2}$ for some $\rho_1 \rho_2 = \rho$.  Neither factor is isomorphic to $\Z/2\Z$, since \eqref{IsoX'} does not split.   Also, neither factor is isomorphic to $\Mu_2$, as that would force $G_\rho \simeq \Mu_2$.  Hence there is a new composition series for $W$ with no constituent isomorphic to $\Mu_2$ and fewer constituents isomorphic to $\Z/2\Z$.  This contradiction implies that there is a filtration for $W$ with grading $[W_2,W/W_2]$ such that neither $\Mu_2$ nor $\Z_2$ is a constituent of the induced filtration on $W_2$.  Also, every constituent of the induced filtration on $W/W_2$ is isomorphic to $\Z/2\Z$.  

Then we can find a closed subgroup scheme $V_2$ of $V$ containing $V_1,$ such that $V/V_2 \simeq W/W_2$ satisfies (ii) and $V_2/V_1 \simeq W_2$ satisfies (iii). 
\end{proof}

Next, we further refine the filtration in Proposition \ref{MoveZMu} when $2 = \pi_1 \pi_2 \pi_3$ is the product of three primes in $\cO$.  Our rearrangement of simple group scheme subquotients is based on Table \ref{Split}.

\begin{table}[h]\begin{caption}{Rearranging constituents}  \label{Split} \end{caption}
\renewcommand{\arraystretch}{1.2}
$$
\begin{array}{|c | c | c | c | c | c |}\hline
\text{Item} & G_\rho & G_{\rho'} & P' \text{ ram} & S \text{ split} & \text{Comments} \\ 
\hline
\text{1}&G_{\pi_1\pi_2}&G_{\pi_1\pi_3}&\pi_3&\{\pi_2\}& \text{QE}_2 \\
\text{2}&G_{\pi_1\pi_2}&G_{\pi_1}&1&\{\pi_2\}&h = 1\\
\text{3}&G_{\pi_1\pi_2}&G_{\pi_2\pi_3}&\pi_3&\{\pi_1\}& \text{QE}_1 \\
\text{4}&G_{\pi_1\pi_2}&G_{\pi_2}&1&\{\pi_1\}&h = 1\\
\text{5}&G_{\pi_1\pi_2}&G_{\pi_3}&\pi_3&\{\pi_1,\pi_2\}& \text{Prop. \ref{AltNonSplit}}\\ \hline
\text{6}&G_{\pi_2}&G_{\pi_3}&\pi_3&\{\pi_2\}& \text{QE}_2 \\ 
\text{7}&G_{\pi_2\pi_3}&G_{\pi_3}&1&\{\pi_2\}&h = 1\\
\text{8}&G_{\pi_1}&G_{\pi_3}&\pi_3&\{\pi_1\}& \text{QE}_1 \\
\text{9}&G_{\pi_1\pi_3}&G_{\pi_3}&1&\{\pi_1\}&h =1\\ \hline
\text{10}&G_{\pi_1\pi_3}&G_{\pi_1}&1&\{\pi_3\}&h =1\\
\text{11}&G_{\pi_1\pi_3}&G_{\pi_2\pi_3}&\pi_2&\{\pi_1\}& \text{QE}_1\\
\text{12}&G_{\pi_1\pi_3}&G_{\pi_2}&\pi_2&\{\pi_1,\pi_3\}&\text{Prop. \ref{AltNonSplit}} \\\hline
\text{13}&G_{\pi_2\pi_3}&G_{\pi_2}&1&\{\pi_3\}&h =1\\
\text{14}&G_{\pi_1}&G_{\pi_2}&\pi_2&\{\pi_1\}&\text{QE}_1  \\ \hline
\text{15}&G_{\pi_1}&G_{\pi_2\pi_3}&\pi_2\pi_3&\{\pi_1\}& \text{QE}_1 \\ \hline
\end{array}
$$ 
\renewcommand{\arraystretch}{1}
\end{table}

\vspace{5 pt}

\noindent {\bf About Table \ref{Split}}.   Let $2 = \pi_1 \pi_2 \pi_3$ in $\cO$ and write $\gp_i = \pi_i \cO$ for $i=1,2,3$.  Refer to Proposition \ref{ExtDim} and the notation there.   To limit the possibilities for  $\Ext^1_\cO(G_\rho, G_{\rho'})$, we impose the following {\em quadratic extension hypotheses} on $K$.
\begin{equation} \label{QE}
\begin{array}{l l}
{\bfQE}_1\!: &\text{There is no quadratic extension of } K \text{ split}\\  
                    &\text{over } \pi_1, \text{with conductor dividing } \infty (\pi_2 \pi_3)^2. \\[2 pt]
{\bfQE}_2\!: &\text{There is no quadratic extension of } K \text{ split}\\   
                    &\text{over } \pi_2, \text{with conductor dividing } \infty \pi_3^2. 
\end{array}
\end{equation}

Note that ${\bf QE}_1$ also implies that there is no quadratic extension of $K$ split over $\pi_1$ with conductor dividing either $\infty \pi_2^2$ or $\infty \pi_3^2$.
\begin{enumerate}[\, $\bullet$]
\item Given ${\bf QE}_1$ and {\bf QE}$_2$, Proposition \ref{ExtDim} implies that $\Ext^1_\cO(G_\rho,G_{\rho'}) = 0$ in Table \ref{Split}, excluding items 5 and 12.  Alternatively, in the unramified case, namely when $P' = 1$, we have $\Ext^1_\cO(G_\rho,G_{\rho'}) = 0$ because the narrow class group of $K$ is trivial.  That special case is denoted by $h = 1$ in the Comments column.
\vspace{2 pt}

\item For items 5 and 12, Proposition \ref{AltNonSplit} applies.  Thus, there is a unique non-trivial class $[G]$ in $\Ext^1_\cO(G_\rho,G_{\rho'})$ and $G$ also is a non-trivial extension of $G_{\pi_j \pi_3}$ by $G_{\pi_k}$ when $\{j,k\}=\{1,2\}$.
\end{enumerate}

\medskip

\begin{prop}  \label{sorter} 
Let $2 = \pi_1 \pi_2 \pi_3$ be the product of three primes in $\cO$ and assume both  ${\bf QE}_1$ and ${\bf QE}_2$.  If $V$ is a prosaic group scheme over $\cO$, then there is a filtration
$$
0 \subseteq V_1 \subseteq V_2 \subseteq V_3 \subseteq V_4 \subseteq V_5 \subseteq V_6 \subseteq V_7 \subseteq V,
$$
with quotients filtered entirely by copies of $G_\rho$, according to the following table:
\renewcommand \arraystretch{1.2}
$$
\begin{array} {| c || c | c | c | c | c | c | c | c |}
\hline
 & V_1 & V_2/V_1 & V_3/V_2 & V_4/V_3 & V_5/V_4 & V_6/V_5 &V_7/V_6 & V/V_7 \\
\hline
G_ \rho & \Mu_2 & G_{\pi_1 \pi_2} &  G_{\pi_1 \pi_3} & G_{\pi_1} & G_{\pi_2\pi_3} & G_{\pi_2 }  & G_{\pi_3} & \Z/2\Z \\
\hline
\end{array}
$$
\renewcommand \arraystretch{1}
\end{prop}

\begin{proof}
Proposition \ref{MoveZMu} gives the construction of $V_1$ and $V_7$, with no constituent of $W_1 = V_7/V_1$ isomorphic to $\Mu_2$ or $\Z/2\Z$.  Suppose that $G_{\pi_1 \pi_2}$ is a constituent of $W_1$.  In the first 4 lines of Table \ref{Split}, $\Ext^1_\cO(G_{\pi_1 \pi_2},G_{\rho'}) = 0$, so we can move $G_{\pi_1 \pi_2}$ to the left, past each such $G_{\rho'}$.  

According to item 5 and Proposition \ref{AltNonSplit}, we might encounter a non-trivial extension $G$ of $G_{\pi_1 \pi_2}$ by $G_{\pi_3}$.  In that case, $G$ also is an extension of $G_{\pi_1\pi_3}$ by $G_{\pi_2}$.  This change of filtration clearly does not introduce new constituents isomorphic to $\Mu_2$, $\Z/2\Z$, $G_{\pi_1\pi_2}$ or $G_{\pi_3}$.  As a result, we can create a filtration $0 \subseteq W_1' \subseteq W_1$, where $W_1'$ is filtered entirely by copies of $G_{\pi_1\pi_2}$ and $W_1/W_1'$ has no constituent isomorphic to $\Mu_2$, $\Z/2\Z$ or $G_{\pi_1\pi_2}$.  Then we can find a closed subgroup scheme $V_2$ of $V_7$ containing $V_1$, such that $V_2/V_1 \simeq W_1'$ and $V_7/V_2 \simeq W_1/W_1'$. 

Thus the constituents of $W_2 = V_7/V_2$ belong to the set of simple group schemes 
$$
\{G_{\pi_3},G_{\pi_2}, G_{\pi_2 \pi_3}, G_{\pi_1},  G_{\pi_1\pi_3} \}.
$$
We wish to move all copies of $G_{\pi_3}$ to the right in $W_2$.  This task is dual to the procedure outlined above and uses only $\Ext^1_\cO(G_\rho,G_{\pi_3}) = 0$ for $G_\rho$ in items 6 through 9 of the table.  As a result, we obtain a closed subgroup scheme $V_6$ of $V_7$ containing $V_2$, such that $V_7/V_6$ is filtered entirely by copies of $G_{\pi_3}$ and the constituents of $W_3 = V_6/V_2$ belong to the set 
$$
 \{G_{\pi_1\pi_3}, G_{\pi_1}, G_{\pi_2 \pi_3}, G_{\pi_2} \}.
$$

To move all copies of $G_{\pi_1\pi_3}$ to the right in $W_3$, we first use the vanishing of $\Ext^1_\cO(G_{\pi_1 \pi_3},G_{\rho'})$ given by items 10 and 11 of the table.  However, in item 12, we might encounter a non-trivial extension $G$ of $G_{\pi_1 \pi_3}$ by $G_{\pi_2}$.  In that case, $G$ also is an extension of $G_{\pi_2 \pi_3}$ by $G_{\pi_1}$, as in Proposition \ref{AltNonSplit}.  With these modifications, we obtain a closed subgroup scheme $V_3$ of $V_6$ containing $V_2$, such that $V_3/V_2$ is filtered entirely by copies of $G_{\pi_1 \pi_3}$ and the constituents of $W_4 = V_6/V_3$ belong to the set 
$$
 \{G_{\pi_2},G_{\pi_2 \pi_3},G_{\pi_1}  \}.
$$

The justification for moving all copies of $G_{\pi_2}$ to the right in $W_4$ is dual to the previous paragraph, using $\Ext^1_\cO(G_\rho,G_{\pi_2}) = 0$ for $G_\rho$ in items 13 and 14 of the table.  Thus, we can find a closed subgroup scheme $V_5$ of $V_6$, containing $V_3$ such that $V_6/V_5$ is filtered entirely by copies of $G_{\pi_2}$ and the constituents of $V_5/V_3$ belong to $\{ G_{\pi_1}, G_{\pi_2 \pi_3} \}$.

To complete the proof, use $\Ext^1_\cO(G_{\pi_1},G_{\pi_2 \pi_3}) = 0$, as given by item 15.
\end{proof}

\begin{prop}  \label{bd}
Suppose that $2 = \rho \bar{\rho}$ in $\cO$, with neither $\rho$ nor $\bar{\rho}$ a unit.  Let $V$ be a group scheme over $\cO$ filtered entirely by copies of $G_\rho$ or entirely by copies of $G_{\bar{\rho}}$.  If the maximal abelian $2$-extension $M/K$ unramified outside $\rho$ and $\infty$ is finite, then the exponent of $V$ divides $2 \, \fdeg{M}{K}$.
\end{prop}

\begin{proof}
This approach is suggested by the proof in \cite[Prop. 8.3]{Sch7a}.  By hypothesis, the primes over 2 in $K/\Q$ are unramified, so $\gcd(\rho,\bar{\rho}) = 1$.  Since Galois acts trivially on $G_\rho$, the field of points $L = K(V)$ is a 2-extension of $K$, as is $L' = K(V^\wedge)$.  Furthermore, $LL' = K(V \oplus V^\wedge)$. 

An extension of an \'{e}tale group scheme by an \'{e}tale group scheme at a prime $\gq$ is \'{e}tale.  It follows that among finite primes, only those dividing $\rho$ might ramify in $L/K$ and only those dividing $\bar{\rho}$ might ramify in $L'/K$.  Hence $(L \cap L')/K$ is unramified at all finite places.  Since the narrow class group of $K$ is trivial, we have $L \cap L' = K$ and  
$$
\Gal(LL'/K) \simeq \Gal(L/K) \times \Gal(L'/K).
$$
Let $F$ and $F'$ respectively, be the maximal subfields of $L$ and $L'$ abelian over $K$.  Then the maximal subfield of $LL'$ abelian over $K$ is $FF'$ and
$$
\Gal(FF'/K) \simeq \Gal(F/K) \times \Gal(F'/K).
$$
Let $2^n$ be the common exponent of $V$, $V^\wedge$ and $V \oplus V^\wedge$.  

Thanks to the Weil pairing, $K(\Mu_{2^n})$ is contained in $FF'$.  The primes over 2 are unramified in $K/\Q$, so each of them is totally ramified in $K(\Mu_{2^n})/K$.  Thus $K \cap \Q(\Mu_{2^n}) = \Q$ and
$$
\Gal(K(\Mu_{2^n})/K) \xrightarrow{\sim} \Gal(\Q(\Mu_{2^n})/\Q).
$$ 
Let $\gP'$ be a prime of $FF'$ over $\gP$ in $F$ and let $\gP$ lie over $\gp$ in $K$.  If $\gp$ is a prime divisor of $\rho \cO$, the order of the inertia group $\cI_{\gP'}(FF'/K)$ therefore is a multiple of $2^{n-1}$.  Since $F'/K$ is unramified over $\gp$, the following restriction map is an isomorphism:
$$
\cI_{\gP'}(FF'/K) \xrightarrow{\sim} \cI_\gP(F/K).
$$
Hence $\vv{\cI_\gP(F/K)}$ is a multiple of $2^{n-1}$.  But $F$ is a subfield of $M$, so the exponent $2^n$ of $V$ divides $2\fdeg{M}{K}$.

Now consider a group scheme $W$ filtered by copies of $G_{\bar{\rho}}$.  Then $V = W^\wedge$ is filtered by copies of $G_\rho$.  We have shown that the exponent of $V$ divides $2\fdeg{M}{K}$.  But $W$ has the same exponent.
\end{proof}

\begin{cor} \label{bdCor}
Let $\gp$ be a prime divisor of $2\cO$ and let $\gc = 2 \gp^{-1}$.  Assume that the maximal abelian $2$-extension $M/K$ ramified only over $\gc$ and $\infty$ is finite. Let $2 = \rho \bar{\rho}$ in $\cO$, where neither $\rho$ nor $\bar{\rho}$ is a unit.  If $V$ is filtered entirely by copies of $G_\rho$ or copies of $G_{\bar{\rho}}$, then the exponent of $V$ divides $2 \,\fdeg{M}{K}$.
\end{cor}

\begin{proof}
Suppose first that $\rho$ is in $\gc$.  Since the maximal abelian 2-extension of $K$ unramified outside $\rho$ and $\infty$ is a subfield of $M$, it too is a finite extension of $K$.  Thus, the exponent of $V$ divides $2 \, \fdeg{M}{K}$ by Proposition \ref{bd}.  If $\rho$ is not in $\gc$, then $\bar{\rho}$ is in $\gc$ and the same conclusion holds.
\end{proof}

\section{Some number theory} \label{ClassField}

Let $F$ be a number field with trivial narrow class group.   For the start of this section only, let $\cO = \cO_F$ denote the ring of integers of $F$ rather than $K$, allowing ramification at primes over 2 in $F/\Q$.  Then we return to the ground field $K$ satisfying Hypothesis $\bfK$, to draw some consequences of \S \ref{Font} and \S \ref{GroupTheory} for the maximal solvable extension of $K$ in the Abrashkin-Fontaine extension $\cF/K$.

Fix a {\em square-free} divisor $\rho$ of $2$ in $\cO$ and let $F^{(\rho)}$ be the maximal abelian 2-extension of $F$ unramified outside the primes over $\rho$ and the archimedean places.  For $n \ge 1$, let $F_n$ be subfield of $F$ of conductor $\infty \rho^n$, so that $F^{(\rho)} = \bigcup F_n$.  We use the class field theoretic description of $\Gal(F^{(\rho)}/F)$ to obtain a test for finiteness of $\fdeg{F^{(\rho)}}{F}$ suitable for the computations in \S \ref{HowDone}.

Denote the completion of $F$ at the prime $\pi$ over 2 by $F_\pi$, write $d_\pi = \fdeg{F_\pi}{\Q_2}$ for the local degree and let $e_\pi$ be the ramification index of $F_\pi/\Q_2$.  Let $U_\pi$ be the group of units in the ring of integers $\cO_\pi$ of $F_\pi$.  For $n \ge 1$, set
$$
U_\pi^{(n)} =1+ \pi^n \cO_\pi \quad \text{and} \quad U_\rho^{(n)} = \prod_{\pi \vert \rho} U_\pi^{(n)}.
$$

Let $C = \cO^\times_{pos}$ denote the group of global units of $F$ that are positive at all real places and let $i\!: \,C \to U_\rho = \prod_{\pi \vert \rho} U_\pi$ be the diagonal map.  For $n \ge 1$, let 
$$
C_n = \{ u \text{ in } \cO^\times_{pos} \, \mid \, u \equiv 1 \hspace{3 pt} (\bmod{\,} \rho^n) \, \}
$$
and let $\ov{i(C_1)}$ be the closure of $i(C_1)$ in $U_\rho^{(1)}$.

\begin{prop} \label{NoThyProp}
There is an isomorphism $\Gal(F^{(\rho)}/F) \simeq U_\rho^{(1)}/\ov{i(C_1)}$ and
\begin{equation}  \label{TowerRank}
\rk_{\Z_2} \Gal(F^{(\rho)}/F) = \sum_{\pi \vert \rho} d_\pi - \rk_{\Z_2} \ov{i(C_1)}.
\end{equation}
Let $e = \max \{e_\pi \, : \,  \pi \vert \rho\}$.  The extension $F^{(\rho)}/F$ is finite if and only if there is some $n > e $ with $F_{n+e} = F_n$. 
\end{prop}

\begin{proof}
Let $k_\pi$ be the residue field at the place $\pi$ over $\rho$.  Then $k_\pi^\times$ has odd order, $U_\pi^{(1)}$ is a pro-2 group of rank $d_\pi$ and the exact sequence
$$
1 \to U_\pi^{(1)} \to U_\pi \to k_\pi^\times \to 1
$$
splits.  By assumption, the narrow class group of $F$ is trivial, so class field theory gives an isomorphism between $\Gal(F^{(\rho)}/F)$ and $U_\rho^{(1)}/\ov{i(C_1)}$.  Thus \eqref{TowerRank} hold and $\Gal(F_n/F) \simeq U_\rho^{(1)}/(U_\rho^{(n)} \, i(C_1))$.

Suppose that $F^{(\rho)}/F$ is finite and take $n$ sufficiently large that the finite part of the conductor of $F^{(\rho)}/F$ divides $\rho^n$.  Thus  $F^{(\rho)}$ is contained in $F_n$.  The reverse inclusion holds by definition of $F^{(\rho)}$ and so $F^{(\rho)} = F_n$.  Indeed, $F^{(\rho)} = F_m$ for all $m > n$.

For the converse, suppose that $m >  n$, so that $C_m$ is contained in $C_n$.   In the following exact diagram, we use the notation $i$ for various maps it induces:
\begin{equation*} \label{snake}
\begin{tikzcd}[column sep = small, row sep = small]
1 \arrow[r] & C_1/C_m \arrow["i",r] \arrow[d]&  U_\rho^{(1)}/U_\rho^{(m)}  \arrow[r] \arrow[d] &   \Gal(F_m/F)  \arrow{r} \arrow[d]  & 1\\
1 \arrow[r] & C_1/C_n  \arrow["i", r] & U_\rho^{(1)}/U_\rho^{(n)}  \arrow[r] & \Gal(F_n/F) \arrow[r] & 1
\end{tikzcd}
\end{equation*}
The vertical arrows are surjective and their kernels give the exact sequence
$$
1 \to C_n/C_m \xrightarrow{i} U_\rho^{(n)}/U_\rho^{(m)} \to \Gal(F_m/F_n) \to 1.
$$

Assume that $F_{n+e} = F_n$ for $n > e$ and set $m=n+e$.  Then $U_\rho^{(n)} = U_\rho^{(m)} \, i(C_n)$.   But $U_\rho^{(m)} = (U_\rho^{(n)})^2$, since the same holds for each prime $\pi$ dividing $\rho$ by the binomial theorem.  It follows by induction that $U_\rho^{(n)} = (U_\rho^{(n)})^{2^j} \, i(C_n)$ for all $j \ge 1$ and so $U_\rho^{(n)}$ is contained in the closure $\ov{i(C_n)}$.  Since $U_\rho^{(n)}$ has finite index in $U_\rho^{(1)}$, we conclude from the class field theory isomorphism in the Theorem that $F^{(\rho)}/F$ is a finite extension.
\end{proof}

\vspace{5 pt}

We  return to a field $K$ satisfying Hypothesis \bfK.  Refer to Notation \ref{cFL1L2} for the Abrashkin-Fontaine extension $\cF/K$ and its subfields $L_1$, $L_2$, \dots.  See Notation \ref{Kp} for the field $K^\gp$.  For quadratic extensions $M/K$, let $P_M$ be the product of the distinct prime divisors of $2 \cO_M$ and let $R_M$ be the ray class group over $M$ of conductor $\infty P_M^2$. 

\begin{prop}  \label{DpTest1}
Assume that $\Gal(L_1/K)$ is a $2$-group and that $R_M$ is a $2$-group for each quadratic extension $M$ of $K$ in $L_1$.  Then $\Gamma = \Gal(L_2/K)$ is a $2$-group.  \begin{enumerate}[{\rm i)}]
\item If $K^\gp = K$ for some prime divisor $\gp$ of $2\cO$, then $L_1$ is the maximal solvable extension of $K$ in $\cF$.  \vspace{2 pt}

\item If $\fdeg{K^\gp}{K} = 2$ for some prime divisor $\gp$ of $2\cO$ and $\gp$ is inert in $K^\gp/K$, then $L_2$ is the maximal Galois $2$-extension of $K$ in $\cF$.  Moreover, $L_2$ is the ray class field of conductor $\infty P^2$ over $K^\gp$, where $P$ is the product of the distinct primes over $2$ in $K^{\gp}$.  
\end{enumerate}
\end{prop}

\begin{proof}
The commutator $\Gamma' = \Gal(L_2/L_1)$ is abelian.  Given that $\Gamma^{ab} = \Gal(L_1/K)$ is a 2-group, it is an elementary 2-group by Proposition \ref{Elem2}.  If an odd prime $p$ divides $\fdeg{L_2}{L_1}$, Proposition \ref{OddDihedral} implies that there is a dihedral extension $D/K$ of degree $2p$.  Let $M$ be the unique subfield of $D$ quadratic over $K$.  Then $D/M$ is cyclic of degree $p$ and at most tamely ramified, i.e., of conductor dividing $\infty P_M$.  But we assume that $R_M$ is a 2-group, so there is no such extension of $M$.

If $K^\gp = K$, then the maximal Galois 2-extension $\cF_2$ of $K$ inside $\cF$ is equal to $L_1$ by Proposition \ref{D1Type}.  But no odd prime divides $\fdeg{L_2}{L_1}$, so $L_2 = L_1$ is the maximal solvable extension of $K$ in $\cF$.  Similarly, in case (ii), $\cF_2 = L_2$ by Proposition \ref{D2Type}(ii).  In addition, Proposition \ref{D2Type}(iii) describes $L_2/K^\gp$ as a ray class extension.
\end{proof}

Suppose that $\fdeg{K^\gp}{K} = 2$ and let $N$ be a quadratic extension of $K^\gp$ in $L_2$.  Consistent with previous notation, let $P_N$ be the product of the distinct prime divisors of $2 \cO_N$ and let $R_N$ be the ray class group over $N$ of conductor $\infty P_N^2$.

\begin{prop}  \label{DpTest2}
Suppose that $\fdeg{K^\gp}{K}=2$, with $\gp$ inert in $K^\gp/K$ and that $\Gal(L_2/K)$ is a non-abelian $2$-group.  If $R_N$ is a $2$-group for each quadratic extension $N$ of $K^\gp$ in $L_2$, then $L_2$ is the maximal solvable extension of $K$ in $\cF$.  
\end{prop}

\begin{proof}
Proposition \ref{D2Type} provides the following information:
\begin{enumerate}[i)]
\item $L_2$ is the maximal Galois 2-extension of $K$ in $\cF$, so $\fdeg{L_3}{L_2}$ is odd.  \vspace{2 pt}

\item The inertia group $\cI_\gP(L_2/K)$ at each prime $\gP$ over $\gp$ in $L_2$ has index 2 in $\Gal(L_2/K)$ and exponent 2, with fixed field $K^\gp$. 
\end{enumerate} 
If an odd prime $p$ divides $\fdeg{L_3}{L_2}$, then Proposition \ref{My9Lem} implies that there is a dihedral extension $F/K^\gp$ of degree $2p$ with $F$ contained in $L_3$.  Let $N$ be the unique quadratic extension of $K^\gp$ in $F$.  Since the ray class modulus of $F/N$ divides $\infty P_N$, we find that $p$ divides $\vv{R_N}$, a contradiction.  Thus $L_3 = L_2$ and $L_2$ is the maximal solvable extension of $K$ in $\cF$.
\end{proof}

The bound \eqref{cFRamBd} on higher ramification  in the definition  of $\cF$ implies that its root discriminant satisfies $\delta_\cF < 4 \delta_K$ where $\delta_K$ is the root discriminant of $K$.  Under GRH, \cite{Odl} gives an upper bound $\bfOd_K$ for $\fdeg{\cF}{\Q}$.  In favorable circumstances, $\bfOd_K$ can be combined with Propositions \ref{DpTest1} and \ref{DpTest2} to ascertain that $\cF$ is a 2-extension of $K$.  When computational difficulties preclude the use of Proposition \ref{DpTest2}, a weaker result still holds.

\begin{prop} \label{cF2Ext}
Assume that $L_2$ is a $2$-extension of $K$.   If one of the following additional conditions applies, then $\cF = L_2$.
\begin{enumerate}[{\rm i)}] 
\item $L_2$ is the maximal Galois $2$-extension of $K$ in $\cF$ and $\bfOd_K < 9 \, \fdeg{L_2}{\Q}$, or \vspace{2 pt}
\item $L_2$ is the maximal solvable extension of $K$ in $\cF$ and $\bfOd_K < 60 \, \fdeg{L_2}{\Q}$.
\end{enumerate}
\end{prop}

\begin{proof}
In case (i), $\cF/L_2$ is a 2-extension by Lemma \ref{9lem}, so $\cF = L_2$.  In case (ii), $\bfOd_K$ forces $\cF/L_2$ to be solvable, so $\cF = L_2$.  
\end{proof}

To satisfy the hypotheses for many of our results, the local behavior of certain extensions of $K$ is severely restricted.   We address some necessary conditions for those restrictions to hold.  Recall that $r_1$ and $r_2$ are the number of real and complex places of $K$.

\begin{prop} \label{dpi}
Let $\pi$ be a prime dividing $2$ in $\cO$ and let $d_\pi = \fdeg{K_\pi}{\Q_2}$ be its local degree.  Set $\gp = \pi \cO$ and  $\gc = 2\gp^{-1}$.
\begin{enumerate}[{\, \rm i)}]
\item If $K^{\gp} = K$, then $d_\pi \ge r_1+r_2$.   \vspace{2 pt}
\item If there is no quadratic extension of $K$ of conductor dividing $\infty \gc^2$, split over $\gp$, then $d_\pi \ge r_1+r_2-1$.   \vspace{2 pt}
\item  If the maximal abelian $2$-extension $F^{(\pi)}$ of $K$ unramified outside $\gp$ and $\infty$ is finite, then $d_\pi \le r_1+r_2-1$.
\end{enumerate}
\end{prop}

\begin{proof}
Item (i) follows from a minor modification of the proof of Proposition \ref{Elem2}.  Let $U_\pi= \cO_\pi^\times$ be the unit group in the completion $\cO_\pi$.  By definition, the ray class conductor of $K^\gp/K$ divides $\infty \gc^2$. Because the narrow class group of $K$ is trivial, Kummer generators (well-defined modulo squares) for the elementary 2-extension $K^\gp/K$ can be chosen from $\cO^\times$, provided they generate an unramified extension over the completion $K_\pi$.  Thus, they are given by the kernel $\Phi$ of the map $\nu$ in the left exact sequence
$$
1 \to \Phi \to \cO^\times/\cO^{\times 2} \xrightarrow{\nu} U_\pi/(1+4\cO_\pi)U_\pi^2.
$$
By the Dirichlet unit theorem, $\rk_{\Z/2\Z} \cO = r_1+r_2$, including the cyclic group of 2-power roots of unity in $\cO$.  Since $\fdeg{U_\pi}{(1+4\cO_\pi)U_\pi^2}  = d_\pi$, we find that
$$
\rk_{\Z/2\Z} \Phi \ge r_1+r_2 - d_\pi.
$$
For $K^\gp = K$ to hold, $\Phi $ must be trivial and so $d_\pi  \ge r_1+r_2$.  The same argument verifies item (ii), replacing the codomain of $\nu$ by $U_\pi/U_\pi^2$, for splitting over $\gp$.

Item (iii) is a consequence of \eqref{TowerRank} with $\rho = \pi$.  Since $\fdeg{\cO^\times_{pos}}{C_1}$ is odd, we have
$$
\rk_{\Z_2} (C_1 \otimes \Z_2)=  \rk_{\Z_2} (\cO_{pos}^\times \otimes \Z_2)  \le \rk_{\Z_2} (\cO^\times \otimes \Z_2) = r_1+r_2-1.
$$
It follows that if $F^{(\pi)}/K$ is finite, then $d_\pi \le r_1+r_2-1$.   
\end{proof}
 
\section{Prosaic abelian varieties}  \label{AboutAbVar} 

Throughout this section, $K$ satisfies Hypothesis {\bf K} and $A$ is a {\em prosaic} abelian scheme over the ring of integers $\cO$ of $K$.  Using the properties of group schemes developed in \S \ref{GrpSchemes}, we obtain conditions on $K$ under which no such $A$ exists.  

See Notation \ref{cFL1L2} for the Fontaine-Abrashkin extension $\cF$ of $K$, large enough to contain the field of points $K(E)$ for every group scheme $E$ over $\cO$ of exponent 2.  The results of \S \ref{ClassField} imply that for certain $K$, the field $\cF$ is a 2-extension of $K$.  If so, every abelian scheme over $\cO$ is prosaic and so the non-existence results obtained here hold for all abelian schemes over $\cO$.  (By convention, the trivial abelian variety will be omitted from consideration.)  In \S \ref{HowDone}, we describe the algorithm that was used to find fields $K$ for which these non-existence criteria apply.

We first recall the argument for bounding the exponent of a group scheme filtered entirely by copies of $\Z/2\Z$ or entirely by copies of $\Mu_2$.

\begin{lem} {\rm \cite[Prop. 5.1]{Sch7a}.} \label{ConstBd} 
For fixed $K$ and $A$, let $V$ and $V'$ be subquotients of $A[2^n]$, with $V$ filtered by copies of $\Z/2\Z$ and $V'$ filtered by copies of $\Mu_2$.  Then $V$ is constant and $V'$ is diagonalizable.  There are bounds independent of $n$ for the exponents of $V$ and $V'$. 
\end{lem}

\begin{proof} 
Since $V$ is filtered by copies of $\Z/2\Z$, it is \'etale, so $K(V)/K$ is 2-extension unramified over all finite places of $K$.  But the narrow class group of $K$ is trivial, so $K(V) = K$.  Hence $V$ is constant. By assumption, there is a filtration 
$$
X' \subset X \subset A[2^n]
$$ 
such that $V = X/X'$.   Thus, $V$ is a closed subgroup scheme of the isogenous abelian variety $B = A/X'$.   Let $\F$ be the residue field of $\cO$ at a prime $\gq$ not dividing 2 and let $\bar{A}$ and $\bar{B}$ be the respective reductions modulo $\gq$.  The reduction map $V \to \bar{V}$ is injective on the 2-group $V$ and isogenous varieties have the same number of points over $\F$.  Hence 
$$
\vv{V} = \vv{\bar{V}} \le \vv{\bar{B}(\F)}=\vv{\bar{A}(\F)},
$$  
so the Weil bound controls $\vv{V}$, independent of $n$.  Since the rank of $A[2^n]$ is $2 \dim A$, the exponent of $V$ also is bounded.  Analogous claims hold for $V'$ by duality.
\end{proof}

The following non-existence result is essentially well-known.

\begin{theo} {\rm \cite[Thm. 2.1]{Sch3}.} \label{1PrimeThm}
Assume that $2$ is prime in $K/\Q$.  Then there is no prosaic abelian scheme over $\cO$.
\end{theo}

\begin{proof} 
By \cite{TO}, the only group schemes of order 2 over $\cO$ are $G_1 = \Z/2\Z$ and $G_2 = \Mu_2$.  Suppose that $A$ is a prosaic abelian scheme over $\cO$.  By Proposition \ref{MoveZMu},  there is a filtration 
$
0 \subseteq V_1 \subseteq V= A[2^n]
$ 
in which $V_1$ is filtered entirely by copies of $\Mu_2$ and $V/V_1$ entirely by copies of $\Z/2\Z$.  According to Proposition \ref{ConstBd}, there are bounds independent of $n$ for $\vv{V_1}$ and $\vv{V/V_1}$, so $A = 0$.
\end{proof}

Our task becomes more difficult as the number of primes over 2 in $\cO$ increases and so we assume that there are at most three primes over 2 in $\cO$.  See Notation \ref{Kp} for the field $K^\gp$.  We make the following {\em tower hypothesis}, which can be tested by Proposition \ref{NoThyProp}.  
\begin{equation}  \label{Tower}
\fbox{\parbox{210 pt}
{\begin{tabular}{l l}
$\bfT(\gp)\!:$\!&the maximal abelian 2-extension of  $K$  \\
&unramified outside $\infty$ and $\gp$ is finite.  
\end{tabular}}} 
\end{equation}

\begin{theo} \label{2PrimeThm}    
Assume that $2$ is the product of two primes in $\cO$ and one of the following conditions applies for some prime divisor $\gp$ of $2\cO$:  
\begin{enumerate}[{\rm i)}]
\item $K^{\gp} = K$,  \, or  \vspace{2 pt}
\item $\fdeg{K^{\gp}}{K} = 2$, $\gp$ is inert in $K^{\gp}/K$ and at least one of $\bfT(\gp)$ or $\bfT(2\gp^{-1})$ holds.
\end{enumerate}
Then there is no {\em prosaic} abelian scheme over $\cO$.
\end{theo}

\begin{proof}
Let $\gp = \pi_1 \cO$ and $\gq = 2\gp^{-1} = \pi_2 \cO$ for prime elements $\pi_1,\pi_2$.  The group schemes of order $2$ over $\cO$ are:
$
\Mu_2, \, \Z/2\Z, \, G_{\pi_1}, \, G_{\pi_2}.
$
Let $A$ be a {\em prosaic} abelian variety everywhere good over $K$.  Apply Proposition \ref{MoveZMu} to  $V=A[2^n]$, to obtain a filtration
$
0 \subseteq V_1\subseteq V_2 \subseteq V
$
in which $V_1$ is filtered only by copies of $\Mu_2$.  Also, $V/V_2$ is filtered only by copies of $\Z/2\Z$ and each constituent of $V_2/V_1$ is isomorphic to $G_{\pi_1}$ or $G_{\pi_2}$.

Assuming (i) or (ii), there is no quadratic extension of $K$ of conductor $\infty \gq^2$ in which $\gp$ splits. Then $M = K$ and $s = 1$ in Proposition \ref{ExtDim} with $\rho=\pi_1$ and $\rho'=\pi_2$.  Hence $\Ext^1_\cO(G_{\pi_1},G_{\pi_2})=0$ and so there is a refined filtration
$$
0 \subseteq V_1 \subseteq V' \subseteq V_2 \subseteq V
$$
in which all constituents of $V'/V_1$ are isomorphic to $G_{\pi_1}$ and all constituents of $V_2/V'$ are isomorphic to $G_{\pi_2}$.  According to Proposition \ref{ConstBd}, both $\vv{V_1}$ and $\vv{V/V_2}$ are bounded, independent of $n$.  It remains to show that $\vv{V'/V_1}$ and $\vv{V_2/V'}$ are similarly bounded, so $A = 0$. 

In case (i), both $V'/V_1$ and $V_2/V'$ have exponent 2 by Corollary \ref{z7Cor} and so their orders are bounded.  Suppose instead that (ii) holds.  Under either $\bfT(\gp)$ or $\bfT(\gq)$, the requisite bounds follow from Proposition \ref{bd}.  
\end{proof}

\begin{theo}  \label{3PrimeThm}
If there are three primes over $2$ in $\cO$ and one of the following conditions applies, then there is no {\em prosaic} abelian scheme over $\cO$.   
\begin{enumerate}[{\rm i)}]
\item $K^\gp = K$ for some prime divisor $\gp = \gp_1$ of $2\cO$,  \, or  \vspace{2 pt}
\item The prime divisors in $2\cO = \gp_1 \gp_2 \gp_3$ can be ordered so that $\fdeg{K^{\gp_1}}{K} = 2$, $\gp_1$ is inert in $K^{\gp_1}/K$ and $\gp_2$ does not split in $K^{\gp_1}/K$.  Also, $\bfT(\gq)$ holds for all prime divisors $\gq$ of $2\cO$.
\end{enumerate}
\end{theo}
 
\begin{proof}
Let $\gp_j = \pi_j \cO$ for $j = 1,2,3$.  Recall the quadratic extension conditions $\bfQE_1$ and $\bfQE_2$ of \eqref{QE}, used in Table \ref{Split}.  In case (i),  there is no quadratic extension of $K$ whose conductor divides $\infty (\pi_2\pi_3)^2$, so ${\bf QE}_1$ and ${\bf QE}_ 2$ automatically hold.  In case (ii), $\bfQE_1$ is given.  The ray class field of conductor $\infty \pi_3^2$ over $K$ is contained in $K^{\gp_1}$.  By assumption, $\gp_2$ does not split in $K^{\gp_1}/K$ and so $\bfQE_2$ holds.

Let $A$ be a {\em prosaic} abelian variety everywhere good over $K$.  In both cases, we have the filtration of $V = A[2^n]$ given by Proposition \ref{sorter}.  Let $V_\rho$ denote the subquotient of $V$ filtered entirely by copies of a fixed $G_\rho$.  By Proposition \ref{ConstBd}, there is a bound independent of $n$ for $\vv{V_\rho}$ when $G_\rho$ is isomorphic to $\Mu_2$ or $\Z/2\Z$.  Assume that neither $\rho$ nor $2/\rho$ is a unit.  In case (i), Corollary \ref{z7Cor} asserts that $2V_\rho = 0$.  In case (ii), Corollary \ref{bdCor} gives a bound on the exponent of each $V_\rho$ when $\bfT(\gq)$ is true for all prime divisors $\gq$ of $2 \cO$.  Hence $A = 0$.
\end{proof}

\section{How were the computations done?}  \label{HowDone}

\numberwithin {equation}{section} 
In this section, we give an algorithm to find Fontaine fields $K$ satisfying Hypothesis {\bf K}.  Recall Notation \ref{cFL1L2} for the Abrashkin-Fontaine extension $\cF$ and its subfields $L_j$.  Recall the witness field $K^\gp$ defined in Notation \ref{Kp} and the following alternative conditions imposed on $K^\gp$ in \eqref{D1D2Intro}.
\begin{equation} \label{D1D2}
\fbox{\parbox{225 pt}
{\begin{tabular}{l l}
${\bfD}_1(\gp)\!:$ &$K^\gp=K$ for some $\gp$ dividing $2\cO.$ \\[3 pt]
$\bfD_2(\gp)\!:$   &$[K^\gp:K] =2$,  with $\gp$ inert in $K^\gp/K$ and\\[2 pt]
&${\bfD}_1(\gq)$ is false for all $\gq \vert 2\cO$.
\end{tabular}}}
\end{equation}
For any number field $M$, let $P_M$ be the product of the distinct prime divisors of $2 \cO_M$ and let $R_M$ be the ray class group over $M$ of conductor $\infty P_M^2.$  

We first sieve for fields $K$ such that the maximal solvable extension of $K$ in $\cF$ is a 2-extension.  For $\bfD_1$-fields $K$, this is accomplished after step {\bf 3} and for $\bfD_2$-fields, after step {\bf 7}.   It follows that under the conditions of Theorem \ref{2PrimeThm} or Theorem \ref{3PrimeThm}, there are no prosaic abelian schemes over $\cO$. In step {\bf 8}, we try to use Odlyzko bounds to show that $\cF$ is in fact a 2-extension.   If so, $K$ is a Fontaine field.

\vspace{5 pt}

\begin{enumerate}[{\, \bf 1.}]
\item There are  at most three primes over 2 in $\cO$ and $R_K$ is a 2-group, else reject $K$.  Then $\Gal(L_1/K)$ is an elementary 2-group of rank $r_1+r_2$ by Proposition \ref{Elem2}.

\item $K$ satisfies $\bfD_1(\gp)$ or $\bfD_2(\gp)$ for some prime $\gp \vert 2\cO$. Otherwise, reject $K$.

\item  Under $\bfD_1(\gp),$ the field $L_1$ is the maximal Galois 2-extension of $K$ in $\cF$ by Proposition \ref{D1Type}. For each quadratic extension $M$ of $K$ in $L_1,$ compute its ray class group $R_M$.  If an odd prime divides some $\vv{R_M}$, reject $K$. Otherwise, $L_1$ also is the maximal solvable extension of $K$ in $\cF$ by Proposition \ref{DpTest1}(i).
   
\item If $\bfD_2(\gp)$ holds, then compute each quadratic extension $M$ of $K$ in $L_1$ and its ray class group $R_M$.  If an odd prime divides some $\vv{R_M}$, reject $K$. Otherwise, $L_2$ is the maximal Galois 2-extension of $K$ in $\cF$ and $\Gal(L_2/K^\gp) = R_{K^\gp}$ by Proposition \ref{DpTest1}(ii).  In particular, if  $L_2 = L_1$, then $L_1$ is the maximal solvable extension of $K$ in $\cF$.

\item In case of two primes,  the tower hypothesis $\bfT(\gq)$ of \eqref{Tower} must hold for some prime divisor $\gq$ of $2 \cO$. In case of three primes, $\bfT(\gq)$ must hold for all prime divisors $\gq$ of $2 \cO$. Otherwise, reject $K$.

\item If $2\cO=\gp_1\gp_2\gp_3$, the prime factors can be ordered so that ${\bfD}_2(\gp_1)$ is satisfied and $\gp_2$ does not split in $K^{\gp_1}/K.$   Then, Theorem \ref{3PrimeThm}(ii) holds.  Otherwise, reject $K.$

\item  If $\bfD_2(\gp)$ holds and $L_2\supsetneq L_1,$ then for each quadratic extension $N$ of $K^\gp$ in $L_2,$ compute  the ray class group $R_N$.   If an odd prime divides some $\vv{R_N}$, reject $K$.   Otherwise, $L_2$ is the maximal solvable extension of $K$ in $\cF$ by Proposition \ref{DpTest2}.

\item $\cF/K$ is a 2-extension if
\begin{enumerate}[$\bullet$]
\item the bound in Proposition \ref{cF2Ext}(i) holds without checking {\bf 7}, or
\item the bound in Proposition \ref{cF2Ext}(ii) holds and either $K$ satisfies {\bf 7}, or $L_2 = L_1$.
\end{enumerate}
\end{enumerate}

\begin{Rem}
All the fields $K$ of degree at most 8 for which step {\bf 1-6} were verified also satisfied {\bf 7}.  Because the time needed to check step {\bf 7} was prohibitive, it was not used for higher degree fields.
\end{Rem}

\noindent Explanation for the columns in Table \ref{punchline}: 
\begin{enumerate}[\, $\bullet$]
\item {\em deg} refers to the degree $n$ of the fields in that row. 

\item {\em LMFDB} contains the number of fields  of odd discriminant, narrow class number 1 and root discriminant at most 9.5 downloaded from \cite{LMF} in August 2026.  

\item The dagger for degree $n$ in $\{10, 12, 14, 16\}$, indicates that we included some fields $K$ of degree $n,$ not in \cite{LMF}, satisfying Hypothesis {\bf K}, such that $K$ is  the narrow Hilbert class field of a field  in \cite{LMF} of smaller degree.

\item The columns {\em $D_i$}  indicate that the corresponding fields satisfy $D_i(\gp)$ for some prime $\gp$ of $\cO$ over 2. Note that if there is only one prime over 2 in $K$, then $D_1$ holds automatically.

\item {\em S} refers to the number of fields $F$ of degree $n$ which do not satisfy Hypothesis {\bf K}, but whose narrow Hilbert class field is a Fontaine field,  so that $F$   is {\em a fortiori} also a Fontaine field. 

\item  {\em Total} refers to the number of Fontaine fields of degree $n$ that were obtained.
\end{enumerate}
 
\begin{table}[h] 
\begin{center}{\begin{caption}{Counting the Fontaine fields we found} \label{punchline} \end{caption} }
\end{center}
 \renewcommand{\arraystretch}{1.2}
 \begin{tabular}{ |c| c | c | c | c | c |c|c||c|}
\hline
\multicolumn{2}{|c|}{}&\text{1 prime}&\multicolumn{2}{|c|}{ \text{2 primes}}&\multicolumn{2}{|c|} {\text{3 primes}}&\multicolumn{2}{|c|}{}\\ 
\hline
deg&LMFDB &$ D_1$&$D_1$ &$ D_2$&$D_1$&$D_2 $&S&Total\\ 
\hline
2&16&3&1&1&0&0&7&12\\
 \hline
3 &64 &16&4&4&0 &1&5&30\\
\hline
4&337&50 &33&18&3&2&21 &127 \\
\hline
5&885 &146 &83&53&4&5&26 &317 \\
\hline
6&3199&577 &360&125 &23&16&66 & 1167\\
\hline 
7&7392 &1542 &903&331&42&40&0& 2858 \\
\hline
8&4703&1874&1005&124&77&36&10 &3126 \\
\hline
9&837 &348&139&11&0&1 &0 & 499\\
\hline
10&26352$^\dagger$ &11584&6023&476&320&75&0 & 18478\\
\hline
11&22&15&5&2& 0&0 &0& 22\\
\hline
12&302$^\dagger$ &65&105&3 & 3 & 0&0&176\\
\hline
14& 34$^\dagger$   & 9  & 4 &  0  &  0 &   0  &  0 &  13  \\
\hline
16&58$^\dagger$  &   9 & 13  &  0  &  0&  0   & 0  &22   \\
\hline
total&  \multicolumn{7}{c||}{} &26847\\
\hline
\end{tabular} 
\end{table}
\renewcommand{\arraystretch}{1}

\begin{Rem} \label{Tcha}
To compare with \cite{Tch}, we consider only odd discriminants and we only assume GRH for number fields, while he assumes the standard conjectures on $L$-series of abelian varieties.  He lists 936 numbers fields of odd discriminant and degree up to 12. Except for 8 fields which  do not satisfy {\em all} our hypotheses, the others are known to us. This gives mutual support for his assumptions and our methods. 
\end{Rem}


\begin{thebibliography}{12345}   
\bibitem[Ab]{Ab} V. A. Abrashkin, Galois modules of group schemes of period $p$ over the ring of Witt vectors, Math. USSR-Izv. {\bf 31} (1988), no. 1, 1--46. 
\bibitem[BK1]{BK1} A. Brumer and K. Kramer,  Non-existence of certain semistable abelian varieties, Manus. Math. {\bf 106} (2001), 291--304.
\bibitem[BK2]{BK2} A. Brumer and K. Kramer, Semistable abelian varieties with small division fields,  in: Galois Theory and Modular Forms, K.-I. Hashimoto et al., eds., Kluwer (2003), 13--38.  
\bibitem[BK3]{BK3} A. Brumer and K. Kramer, Certain abelian varieties bad at only one prime. Algebra and Number Theory {\bf 12} (2018) 1027-1071.
\bibitem[BK4]{BK4} A. Brumer and K. Kramer, Hyperelliptic $\cS_7$-curves of prime conductor. Manus. Math.  {\bf 171} (2023) 169-213.
\bibitem[Ca]{Ca} F. Calegari, Semistable abelian varieties over $\Q$, Manus. Math. {\bf 113} (2004), 507--529.
\bibitem[CG]{CG} F. Campagna and P. Goodman, Semistable abelian varieties over $\Q$ with bad reduction at 19 only, arXiv:2510.12625.
\bibitem[De]{De} L. Demb\'el\'e, On the existence of abelian surfaces with everywhere good reduction, arXiv:1903.10394. 
\bibitem[Fa]{Fa} G. Faltings, Endlichkeitss\"atze f\"ur abelsche Variet\"aten \"uber Zahlk\"orpern, Invent. Math. 73 (1983), 349--366.
\bibitem[Fo]{Fo} J.-M. Fontaine, Il n'y a pas de vari\'{e}t\'{e} ab\'{e}lienne sur $\Z$, Invent. Math. {\bf 81} (1985), 515--538.
\bibitem[Kha]{Kha} C. Khare, Modularity of Galois representations and motives with good reduction properties, J. Ramanujan Math. Soc. {\bf 22} (2007) 1–26.
\bibitem[LMF]{LMF} The LMFDB Collaboration, The L-functions and modular forms database, \\ \url{https://www.lmfdb.org}, 2026.
\bibitem[MAG]{MAG} W. Bosma, J. Cannon  and C. Playoust. The Magma algebra system. I. The user language. J. Symb. Comp. {\bf 24} (1997), 235--265.
\bibitem[Mes]{Mes} J.-F. Mestre, Formules explicites et minorations de conducteurs de vari\'et\'es alg\'ebriques, Compos. Math. {\bf 58} (1982), 209--232.   
\bibitem[Odl]{Odl} A. Odlyzko, Lower bounds for discriminants of number fields, II, T\^{o}hoku Math. Jl. {\bf 29} (1977), 209--216, \url{http://www.dtc.umn.edu/~odlyzko/unpublished/index.html}.
\bibitem[Ray]{Ray} M. Raynaud,  Sch\'emas en groupes de type $(p,p,\dots,p),$ Bull. Soc. Math. France, {\bf 102} (1974), 241--280.  
\bibitem[Sch1]{Sch1} R. Schoof,  Abelian varieties over $\Q(\sqrt{6})$   with good reduction everywhere. In Class Field Theory - its Centenary and Prospect, Adv. Studies in Pure Mathematics, Tokyo 2001. 
\bibitem[Sch2]{Sch2} R. Schoof,  Abelian varieties over the field of the 20th roots of unity with good reduction  everywhere, in Applications of Algebraic Geometry to Coding Theory, Physics and Computation, Kluwer, Dordrecht (2001), 291-296. 
\bibitem[Sch3]{Sch3} R. Schoof,  Abelian varieties over cyclotomic fields   with everywhere good reduction, Math. Ann. {\bf 325} (2003), 413--448. 
\bibitem[Sch4]{Sch4} R. Schoof,  Abelian varieties over $\Q$ with bad reduction in one prime only, Compos. Math. {\bf 141} (2005), 847--868.
\bibitem[Sch5]{Sch5} R. Schoof, Semistable abelian varieties with good reduction outside 15,  Manus. Math. {\bf 139} (2012), 49--70.
\bibitem[Sch6]{Sch6} R. Schoof, On the modular curve $X_0(23)$, Geometry and Arithmetic, EMS Publishing House, Z\"urich 2012, 317--345.
\bibitem[Sch7]{Sch7} R. Schoof, Abelian varieties over real quadratic fields with good reduction everywhere, Private communication, July 2020. 
\bibitem[Sch7a]{Sch7a} R. Schoof, Abelian varieties over real quadratic fields with good reduction everywhere, preprint, schoof/papers.html.
\bibitem[Ser]{Ser} J.-P. Serre, Local Fields, Lecture Notes in Math. {\bf 67}, Springer-Verlag, 1979.
\bibitem[TO]{TO} J. Tate and F. Oort, Group Schemes of Prime Order,  Ann. scient. \'Ec. Norm. Sup., {\bf 3} (1970), 1--21.
\bibitem[Tch]{Tch} P. Tchamitchian, Explicit lower bounds on the conductors of elliptic curves and abelian varieties over number fields, arXiv:2410.12922. 
\bibitem[Zar]{Zar} J.G. Zarhin, A remark on endomorphisms of abelian varieties over function fields of finite characteristic, Math USSR {\bf 8}(1974), 477--480.
\end{thebibliography}
\end{document}